\documentclass{amsart} % For Latex2e
\usepackage[active]{srcltx}
\usepackage{bm}
\newtheorem{definition}{Definition}[section]

\newtheorem{remark}[definition]{Remark}
\newtheorem{example}[definition]{Example}

\newtheorem{lemma}[definition]{Lemma}
\newtheorem{proposition}[definition]{Proposition}
\newtheorem{theorem}[definition]{Theorem}
\newtheorem{corollary}[definition]{Corollary}

\newtheorem{algorithm}[definition]{Algorithm}

\def\P{{\mathbb{P}}}
\def\V{{\mathbb{V}}}

\def\K{{\mathbb{K}}}

\def\N{{\mathbb{N}}}

\def\cB{{\mathcal{B}}}

\def\bfB{{\mathbf{B}}}
\def\bfn{{\mathbf{w}}}
\def\bfn{{\mathbf{n}}}
\def\bfU{{\mathbf{U}}}
\def\bfV{{\mathbf{V}}}

\def\bfu{{\mathbf{u}}}
\def\bfv{{\mathbf{v}}}

\def\cU{{\mathcal{U}}}

\def\cF{{\mathcal{F}}}

\def\Z{{\mathbb{Z}}}

\def\M{{\mathbb{M}}}

\begin{document}
\title[Unattainable points for  Rational Hermite Interpolation]{The Set of Unattainable points for the Rational Hermite Interpolation  Problem}

\author{Teresa Cortadellas Ben\'itez}
\address{Universitat de Barcelona, Facultat de Educaci\'o.
Passeig de la Vall d'Hebron 171,
08035 Barcelona, Spain}
\email{terecortadellas@ub.edu}

\author{Carlos D'Andrea}
\address{Universitat de Barcelona, Departament de Matem\`atiques i Inform\`atica,
 Universitat de Barcelona (UB),
 Gran Via de les Corts Catalanes 585,
 08007 Barcelona,
 Spain} \email{cdandrea@ub.edu}
\urladdr{http://www.ub.edu/arcades/cdandrea.html}

\author{Eul\`alia Montoro}
\address{Universitat de Barcelona, Departament de Matem\`atiques i Inform\`atica,
 Universitat de Barcelona (UB),
 Gran Via de les Corts Catalanes 585,
 08007 Barcelona,
 Spain}
\email{eula.montoro@ub.edu}

\thanks{Cortadellas and D'Andrea  are supported by the Spanish MEC research project MTM2013-40775-P, Montoro by the Spanish MINECO/FEDER research project MTM 2015-65361-P }
\date{\today}

\subjclass[2010]{Primary 14Q15  ; Secondary 13P05, 68W30}

\keywords{Rational Hermite Interpolation, Unattainable Points, Equidimensional Varieties, Rational Varieties, Complete Intersections, Structured Matrices}

\begin{abstract}
We describe geometrically and algebraically  the set of unattainable points for the Rational Hermite Interpolation Problem (i.e. those points where the problem does not have a solution). We show that this set is a union of equidimensional complete intersection varieties of odd codimension, the number of them being equal to the minimum between the degrees of the numerator and denominator of the problem. Each of these equidimensional varieties can be further decomposed as a union of as many rational (irreducible) varieties as input data points. We exhibit algorithms and equations defining all these objects.
\end{abstract}
\maketitle

\section{Introduction}
Let $\K$ be a field, $k,\,l,\, n_1,\ldots, n_l\in\Z_{>0}$ with $k\leq n:= n_{1}+\ldots+n_{l}.$ For  $u_1,\ldots, u_l\in \K$ with $u_i\neq u_j$ if
$i\neq j,$ and $v_{i,j}\in\K$ with $i=1,\ldots, l,\,j=0,\ldots, n_i-1,$  The  Rational Hermite Interpolation Problem (RHIP) associated with this data as stated in \cite{sal62, kah69, sal84,SW91}, is the
following:  decide if there exist -and if so compute- polynomials
$A(x),\,B(x)\in\K[x]$ of degrees bounded by $k-1$ and $n-k$ respectively such that
$B(u_i)\neq0$ for all $i=1,\ldots, l,$ and
\begin{equation}\label{hermite}
\left(\frac{A}{B}\right)^{(j)}(u_{i})=j!v_{i,j},\,i=1,\ldots, l,\,j=0,\ldots,n_{i}-1.
\end{equation}
The factorial in the equation above is introduced to simplify some of the formulas that will appear later.  For \eqref{hermite} to be as general as possible,  we need to impose Char$(\K)=0$ or Char$(\K)\geq\max\{n_1,\ldots, n_l\}.$  When $l=n$ (i.e. $n_1=\ldots=n_l=1$), the RHIP coincides with the classical Rational Interpolation Problem \cite{cau41, wuy75, SW86}.
If  $k=n,$ the RHIP descends to the well-known Hermite Interpolation Problem. But in contrast with it, there is not always a solution for the RHIP for any input data. For instance, if we set $k=2,\,l=2,\,n_{1}=2,\,n_{2}=1,\,u_1=1,\, u_2=2,\,v_{1,0}=1,\,v_{1,1}=0, v_{2,0}=0,$ one can check
straightforwardly that there is no solution for  \eqref{hermite}, see Example \ref{beow} below.

The standard approach to this problem from both an algorithmic and complexity point of view  is via the Extended Euclidean Algorithm as it is explained in \cite{vzGJ13} (see also \cite{ant88}, and \S \ref{EAR} in this text). There are also alternative approaches by using structured matrices (\cite{VBB92, BL00}), barycentric coordinates (\cite{SW86, SW91}), orthogonal polynomials (\cite{EK89,gem93}), and computation of syzygies (\cite{rav97}). Barycentric coordinates seem to be stable when working with approximate data, but not very fast, while the use of orthogonal polynomials is efficient thanks to the use of Jacobi's method for inverting matrices, but their results  are limited to the rational interpolation problem (without multiplicities) only. Parametric representations of the solutions in general situations can be found in \cite{las03, DKS15}.

In all the previous results, optimal bounds of complexity are achieved, and parametric expressions for $A(x)$ and $B(x)$ are given when they exist, but an explicit description of the set of the so called 
unattainable points for the RHIP, i.e. the set of data $\{u_i, v_{i,j}\}$ such that the RHIP does not have solutions, cannot be obtained straightforwardly from these approachs. The purpose of this paper 
is to  characterize them both geometrically and algebraically. Our main result, given in Theorems \ref{mtt} and \ref{mts} states that the set of ill-posed point is a union of $\min\{k-1,n-k\}$ equidimensional complete intersection varieties of odd codimension. Moreover, each of these l varieties can be further decomposed as a union of $l$  rational (irreducible) varieties. As a by-product, we  will produce explicit expressions for the solution for this problem valid in different regions of the space of parameters, and alternative
algorithms based only in elementary Linear Algebra,  without the need of applying neither barycentric coordinates, nor the Euclidean Division Algorithm. The complexity of solving these problems as well as the extension of  the methods in \cite{EK89} to the RHIP will be the subject of  a future paper.

To deal with the input data properly, we set ${\bf n}=(n_{1},\ldots,n_{l}),\, \bfu=(u_1,\ldots, u_l),\,\bfv_i=(v_{i,0},\ldots, v_{i,n_i-1}), \,1\leq i\leq l,$ and $\bfv=(\bfv_1,\ldots,\bfv_l).$  For 
$j,\,t\in\N,$ we define the $t$-th Pochhammer symbol as follows: $(j)_0=1,\,(j)_t=j\cdot(j-1)\dots (j-t+1).$
The Weak Hermite Interpolation Problem (WHIP), asks to compute polynomials $A_{\bfu,\bfv,\bfn,k-1}(x),\,B_{\bfu,\bfv,\bfn,n-k}(x)\in\K[x]$ of degrees bounded by $k-1$ and $n-k$ respectively
such that
\begin{equation}\label{whermite}
A_{\bfu,\bfv,\bfn,k-1}^{(j)}(u_i)=\sum_{t=0}^{j}(j)_t v_{i,t}\,B^{(j-t)}_{\bfu,\bfv,\bfn,n-k}(u_i),\,i=1,\ldots, l,\,j=0,\ldots,n_{i}-1.
\end{equation}
It is easy to verify that any solution of \eqref{hermite} also satisfies \eqref{whermite}, but not the other way around as $B_{\bfu,\bfv,\bfn,n-k}(x)$ may vanish at some of the $u_i$'s; that is,
if $B(u_{i})\neq 0$ for $i=1,\ldots,l$, then WHIP and RHIP are equivalents. In contrast, given the input
data $(\bfu,\,\bfv)$ as before,  \eqref{whermite} leads to a homogeneous linear system of equations in the coefficients of the
polynomials $A_{\bfu,\bfv,\bfn,k-1}(x),\,B_{\bfu,\bfv,\bfn,n-k}(x)$ of $n$ equations with $n+1$ unknowns, so there is always a non trivial
solution of it.
Indeed, let $\M_{\bfu,\bfv,\bfn,k}\in\K^{n\times(n+1)}$ be the matrix of the linear system \eqref{whermite}. Computing it
explicitly, we have
\begin{equation}\label{symbM}
 \M_{\bfu,\bfv,\bfn,k}=\left(
\begin{array}{c} \M_{u_{1},\bfv_{1},k,n}\\
 \vdots\\
 \M_{u_{l},\bfv_{l},k,n}\end{array}\right),
\end{equation}
where

\begin{equation}
 \M_{u_{i},\bfv_{i},k,n}=\left(\begin{array}{ccc}
  M L_{u_{i},\bfv_{i},k,n}&|&M R_{u_{i},\bfv_{i},k,n}
 \end{array}\right)\in\K^{n_{i}\times(n+1)}
\end{equation}
 and 

\begin{equation}\label{symbML}
M L_{u_{i},\bfv_{i},k,n}=\left(
\begin{array}{ccccc}
1&u_i&u_i^2&\ldots &u_i^{k-1}\\
0&1&2u_i&\ldots &{k-1\choose 1}u_i^{k-2}\\
\vdots&\vdots&\vdots&\ldots&\vdots\\
0&0&0&\ldots &{k-1\choose n_{i}-1}u_i^{k-n_{i}}
\end{array}
\right),
\end{equation}

\begin{equation}\label{symbMR}
M R_{u_{i},\bfv_{i},k,n}=\left(
  \begin{array}{cccc}
-v_{i,0}&-v_{i,0}u_i&\ldots&-v_{i,0}u_i^{n-k}\\
-v_{i,1}&-(v_{i,1}u_{i}+v_{i,0})&\ldots&-\sum_{t=0}^{1}{n-k \choose t}v_{i,1-t}u_{i}^{n-k-t}\\
\vdots&\vdots&\ldots&\vdots\\
-v_{i,n_{i}-1}&-(v_{i,n_{i}-1}u_{i}+v_{i,n_{i}-2})&\ldots&-\sum_{t=0}^{n_{i}-1}{n-k \choose t}v_{i,n_{i}-1-t}u_{i}^{n-k-t}
\end{array}
\right),
\end{equation}
with ${k\choose j}=0$ if $j>k.$ Note that the coefficients of a solution of the WHIP, sorted properly, are elements of the kernel of \eqref{symbM}. We will see in
Proposition \ref{uni} that all the nontrivial pairs $A_{\bfu,\bfv,\bfn,k-1}(x),\,B_{\bfu,\bfv,\bfn,n-k}(x)\in\K[x]$ solving the WHIP,
up to a constant,
produce the same fraction $\frac{A_{\bfu,\bfv,\bfn,k-1}(x)}{B_{\bfu,\bfv,\bfn,n-k}(x)}\in\K(x),$ which will be the solution of the RHIP if
the ``reduced'' fraction is also a solution of the WHIP (Theorem \ref{hipp}). So, finding non trivial solutions of the WHIP ``almost''
solves the RHIP. And it is of interest of course to get the ``minimal'' solution (the one with minimal degree) of this problem, which it is known (see Proposition \ref{2.5}) to be unique up to a constant.
\par
If $\M_{\bfu,\bfv,\bfn,k}$ has maximal rank, then the solution to the WHIP is given by the maximal minors of this matrix.  Otherwise,
some non trivial Linear Algebra must be performed in order to find a non-zero vector in the kernel of this matrix. Our aim is to give an
explicit algebraic formulation in terms of the input data $(\bfu,\bfv)$ which will allow us to produce solutions of the WHIP in all the
cases as functions of minors of suitable matrices. This approach will also give us a full description of the unattainable points for the RHIP.
To do this, we need to consider matrices like \eqref{symbM}, but in a more general setting:  denote with
$U_1,\ldots, U_l,\,V_{1,0},\ldots, V_{1,n_{1}-1},\ldots,V_{l,0},\ldots,V_{l,n_{l}-1}$
a set of $l+n$ indeterminates over $\K,$ and set $\bfU=(U_1,\ldots, U_l),\, \bfV=(\bfV_1,\ldots, \bfV_l).$
For $\alpha,\,\beta,\in\Z_{\geq0},$ and ${\bfn}:=(n_{1},\ldots, n_{l})\in\Z^{l}_{\geq0},$
we set
\begin{equation}\label{symbMM}
\M_{\alpha,\beta,\bf{n}}(\bfU,\bfV)=\left(
\begin{array}{c}
\M_{\alpha,\beta,n_{1}}(U_{1},\bfV_{1})\\
\vdots\\
\M_{\alpha,\beta,n_{l}}(U_{l},\bfV_{l})
\end{array}
\right),
\end{equation}
where for $i=1,\ldots, l,\,
\M_{\alpha,\beta,n_{i}}(U_{i},\bfV_{i})$
stands for
{\tiny
\begin{equation}
=\left(
\begin{array}{ccccccccc}
1&U_i&U_i^2&\ldots &U_i^{\alpha}&-V_{i,0}&-V_{i,0}U_i&\ldots&-V_{i,0}U_i^{\beta}\\
0&1&2U_i&\ldots &{\alpha\choose 1} U_i^{\alpha-1}&-V_{i,1}&-(V_{i,1}U_{i}+V_{i,0})&\ldots&-\sum_{t=0}^{1}{\beta \choose t}V_{i,1-t}U_{i}^{\beta-t}\\
\vdots&\vdots&\vdots&\ldots&\vdots&\vdots&\vdots&\ldots&\vdots\\
0&0&0&\ldots &{\alpha \choose n_{i}-1}U_i^{\alpha-(n_{i}-1)}&-V_{i,n_{i}-1}&-(V_{i,n_{i}-1}U_{i}+V_{i,n_{i}-2})&\ldots
&-\sum_{t=0}^{n_{i}-1}{\beta \choose t}V_{i,n_{i}-1-t}U_{i}^{\beta-t}
\end{array}
\right).
\end{equation}}
For $(\bfu,\,\bfv)\in\K^l\times\K^n,$ note that $\M_{\bfu, \bfv,\bfn, k}=\M_{k-1,n-k,\bfn}(\bfu,\bfv).$
In the case $\alpha+\beta=n-1,$  for $i=1,\ldots, n+1,$ we set
$\Delta^{\bf n}_{\alpha+1,i}$ to be the signed $i$-th maximal minor of $\M_{\alpha,\beta,\bf{n}}(\bfU,\bfV).$
Let $Z\subset\K^l$ be the algebraic set defined by
$$Z=\{(u_1,\ldots, u_l)\in\K^l:\,\prod_{1\leq i<j\leq n}(u_i-u_j)=0\}.$$ Our input $\bfu$ will be an element of $\K^l\setminus Z.$
So, the initial data can be taken from  $\left(\K^l\setminus Z\right)\times\K^n $.
 Our first main result  is the following:
\begin{theorem}\label{mtt}
Let $k, n_1,\ldots, n_l\in\N,\,1\leq k\leq n_1+\ldots+n_l=:n,$ and $\K$ a field with Char$(\K)=0$ or Char$(\K)\geq\max\{n_1,\ldots, n_l\}.$
Set $m:=\min\{k-1,n-k\}.$
The set of unattainable points for the RHIP is a disjoint union $\cB_1\sqcup\cB_3\sqcup\ldots\sqcup \cB_{2m-1}$, where, for $j=1,\ldots, m,\,\cB_{2j-1}$  is
the union of the following $2l$ constructible sets in $(\K^l\setminus Z)\times \K ^n$ with the Zariski topology:
{\small \begin{equation}\label{j}
\begin{array}{c}
\Delta^{\bfn}_{k-j+2,k-j+2}= \dots = \Delta^{\bfn}_{k+j-1,k+j-1} =\sum_{\ell=k-j+1}^{n-2j+2}\Delta^{\bfn}_{k-j+1,\ell+1}U_i^{\ell-k+j-1}= 0,\, \Delta^{\bfn}_{k-j+1,k-j+1}\neq0 \\
\mbox{and} \\
\Delta^{\bfn}_{k-j+2,k-j+2}= \dots = \Delta^{\bfn}_{k+j-1,k+j-1} =\sum_{\ell=k+j-1}^{n}\Delta^{\bfn}_{k+j-1,\ell+1}\,U_i^{\ell-k-j+1}=0,\, \Delta^{\bfn}_{k+j,k+j}\neq0,
\end{array}
\end{equation}}
for $1\leq i\leq l$ (if $j=1$ there is only one equation, and two inequalities above). The two sets in \eqref{j} coincide
in $\{ \Delta^{\bfn}_{k-j+1,k-j+1}\neq0\}\cap\{ \Delta^{\bfn}_{k+j,k+j}\neq0\}.$
If $\K$ is algebraically closed, $\cB_{2j-1}$ is the union of $l$  rational irreducible  varieties  of  codimension $2j-1$ in $(\K^l\setminus Z)\times\K ^n.$
\end{theorem}
This result follows immediately from Theorems \ref{bp} and \ref{finn}. 

\begin{example}\label{beow}
For our example above, we have $k=2,l=2,\bfn=(2,1).$ So, $n=3,$ and  $m=1.$ 
According to Theorem \ref{mtt}, the set of unattainable points coincide with $\cB_1.$ Computing it explicitly  from \eqref{j}, we get 
$$
\cB_1=\Big(\{\Delta^{\bfn}_{2,3}+\Delta^{\bfn}_{2,4}U_1=0\}\cup\{\Delta^{\bfn}_{2,3}+\Delta^{\bfn}_{2,4}U_2=0\}\Big)\cap\Big(\{\Delta^{\bfn}_{2,2}\neq0\}\cup\{ \Delta^{\bfn}_{3,3}\neq0\}\Big),
$$
where 
the minors $\Delta^{\bfn}_{2,j}, j=2,3,4,$ are extracted from 
 $$\M_{1,1,\bfn}(\bfU,\bfV)=\left(\begin{array}{ccccc}1&U_{1}&-V_{1,0}&-V_{1,0}U_{1}\\ 0&1&-V_{1,1}&-V_{1,1}U_{1}-V_{1,0}\\ 1&U_{2}&-V_{2,0}&-V_{2,0}U_{2}
\end{array}\right),$$
while $\Delta^{\bfn}_{3,3}$ is the minor obtained by deleting the third column in
 $$\M_{2,0,\bfn}(\bfU,\bfV)=\left(\begin{array}{ccccc}1&U_{1}&U_1^2&-V_{1,0}\\ 0&1&2U_1&-V_{1,1}\\ 1&U_{2}&U_2^2&-V_{2,0}
\end{array}\right).$$
Computing all these expressions,  we  get
$$\begin{array}{lcl}
\Delta^{\bfn}_{2,2}&=&V_{1,0}^2 - V_{1,0} V_{2,0} - U_1 V_{1,1} V_{2,0} + U_2 V_{1,1} V_{2,0}\\
\Delta^{\bfn}_{3,3}&=& (U_2-U_1)^2\\
\Delta^{\bfn}_{2,3}+\Delta^{\bfn}_{2,4}U_1&=& (V_{2,0}-V_{1,0})(U_2-U_1) \\
\Delta^{\bfn}_{2,3}+\Delta^{\bfn}_{2,4}U_2&=&V_{1,1}(U_1-U_2)^2
\end{array}.$$
From the above we deduce that 
$\cB_1=\{V_{1,0}=V_{2,0},\, V_{1,1}\neq0\}\cup\{V_{1,1}=0,\,V_{1,0}\neq V_{2,0}\},$ a union of two rational (actually linear) varieties. 
In the situation above, we have that  $v_{1,1}=0,$ and hence we conclude that
$(1,2; (1,0),(0))$ is an unattainable point for the RHIP.
\end{example}

\begin{example}
Set $n=5,k=3,l=1,$ so we have $\bfn=(5),$ and $m=2.$  To compute their equations,  the matrices to be considered are
\begin{equation}\label{artilleria}
\begin{array}{l}
{\tiny \M_{1,3,({5})}(\bfU,\bfV)=\left(\begin{array}{cccccl}1&U_{1}&-V_{1,0}&-V_{1,0}U_{1}& -V_{1,0} U_1^2&-V_{1,0}U_1^3 \\
0&1&-V_{1,1}&-V_{1,1}U_{1}-V_{1,0} &-V_{1,1}U_1^2-2V_{1,0}U_1& -V_{1,1}U_1^3-3V_{1,0}U_1^2\\ 
0&0&-V_{1,2}&-V_{1,2}U_{1}-V_{1,1} & -V_{1,2}U_1^2-2V_{1,1}U_1-V_{1,0}&-V_{1,2}U_1^3-3V_{1,1}U_1^2-3V_{1,0}U_1\\
0& 0&-V_{1,3}& -V_{1,3}U_1-V_{1,2}& -V_{1,3}U_1^2-2V_{1,2}U_1-V_{1,1}&-V_{1,3}U_1^3-3V_{1,2}U_1^2-3V_{1,1}U_1-V_{1,0}\\
0&0&-V_{1,4}&-V_{1,4}U_1-V_{1,3}&-V_{1,4}U_1^2-2V_{1,3}U_1-V_{1,2}&-V_{1,4}U_1^3-3v_{1,3}U_1^2-3V_{1,2}U_1-V_{1,1}
\end{array}\right),}\\ 
\\
\M_{2,2,({5})}(\bfU,\bfV)=\left(\begin{array}{ccccccc}1&U_{1}&U_1^2&-V_{1,0}&-V_{1,0}U_{1}& -V_{1,0} U_1^2 \\
0&1&2U_1&-V_{1,1}&-V_{1,1}U_{1}-V_{1,0} &-V_{1,1}U_1^2-2V_{1,0}U_1\\ 
0&0&1&-V_{1,2}&-V_{1,2}U_{1}-V_{1,1} & -V_{1,2}U_1^2-2V_{1,1}U_1-V_{1,0}\\
0&0& 0&-V_{1,3}& -V_{1,3}U_1-V_{1,2}& -V_{1,3}U_1^2-2V_{1,2}U_1-V_{1,1}\\
0&0&0&-V_{1,4}&-V_{1,4}U_1-V_{1,3}&-V_{1,4}U_1^2-2V_{1,3}U_1-V_{1,2}
\end{array}\right), \\
\M_{3,1,({5})}(\bfU,\bfV)=\left(\begin{array}{ccccccc}1&U_{1}&U_1^2&U_1^3& -V_{1,0}&-V_{1,0}U_{1} \\
0&1&2U_1&3U_1^2&-V_{1,1}&-V_{1,1}U_{1}-V_{1,0} \\ 
0&0&1&3U_1&-V_{1,2}&-V_{1,2}U_{1}-V_{1,1} \\
0&0& 0&1&-V_{1,3}& -V_{1,3}U_1-V_{1,2}\\
0&0&0&0&-V_{1,4}&-V_{1,4}U_1-V_{1,3}
\end{array}\right),
\\ \mbox{and}\\
\M_{4,0,({5})}(\bfU,\bfV)=\left(\begin{array}{ccccccc}
1&U_{1}&U_1^2&U_1^3& U_1^4&-V_{1,0} \\
0&1&2U_1&3U_1^2&4U1^3&-V_{1,1}\\ 
0&0&1&3U_1&6 U_1^2&-V_{1,2}\\
0&0& 0&1&4 U_1&-V_{1,3}\\
0&0&0&0&1&-V_{1,4}
\end{array}\right).
\end{array}
\end{equation}
From Theorem \ref{mtt} we know that the set of unattainable points $\cB$ decomposes as $\cB_1\sqcup\cB_3,$ where$$
\cB_1=\{\Delta^{(5)}_{3,4}+\Delta^{(5)}_{3,5}U_1+\Delta^{(5)}_{3,6}U_1^2=0\}\cap\Big(\{\Delta^{(5)}_{3,3}\neq0\}\cup\{ \Delta^{(5)}_{4,4}\neq0\}\Big),
$$
is a rational variety of codimension $1,$ while $\cB_3\subset \{\Delta^{(5)}_{2,2}\neq0\}\cup\{ \Delta^{(5)}_{5,5}\neq0\}$ can be described as follows:
\begin{equation}\label{uuno}\cB_3\cap\{ \Delta^{(5)}_{2,2}\neq0\}=\{ \Delta^{(5)}_{3,3}=\Delta^{(5)}_{4,4}=\Delta^{(5)}_{2,3}+\Delta^{(5)}_{2,4}U_1=0\},
\end{equation}
and 
\begin{equation}\label{ddos}
\cB_3\cap\{ \Delta^{(5)}_{5,5}\neq0\}=\{ \Delta^{(5)}_{3,3}=\Delta^{(5)}_{4,4}=\Delta^{(5)}_{4,5}+\Delta^{(5)}_{4,6}U_1=0\}.
\end{equation}          
We compute explicitly two of the polynomials defining this set:
$$\begin{array}{ccl}
\Delta^{(5)}_{4,5}+\Delta^{(5)}_{4,6}U_1&=&V_{1,3},\\
\Delta^{(5)}_{2,3}+\Delta^{(5)}_{2,4}U_1&=&-V_{1,1}^3 + 2 V_{1,0} V_{1,1} V_{1,2} + U_1^2 V_{1,1} V_{1,2}^2 + 2 U_1^3 V_{1,2}^3 - V_{1,0}^2 V_{1,3} \\
& & - 
 U_1^2 V_{1,1}^2 V_{1,3} - U_1^2 V_{1,0} V_{1,2} V_{1,3} - 4 U_1^3 V_{1,1} V_{1,2} V_{1,3}  + 
 2 U_1^3 V_{1,0} V_{1,3}^2 \\ & &+ U_1^2 V_{1,0} V_{1,1} V_{1,4} + 2 U_1^3 V_{1,1}^2 V_{1,4} - 
 2 U_1^3 V_{1,0} V_{1,2} V_{1,4}.
 \end{array}$$
Even though they are both very different expressions, an explicit computation shows that in  $\{\Delta^{(5)}_{5,5}\neq0\neq\Delta^{(5)}_{2,2}\},$ both \eqref{uuno} and \eqref{ddos} are equal to:
$$V_{1,1}=0,\, V_{1,2}=0,\, V_{1,3}=0,$$
which confirms the claim of Theorem \ref{mtt} for this case. From this latter expression we also deduce that $\cB_3$ is a rational (linear) variety of codimension $3.$
\end{example}

The fact that several minors of different matrices as in \eqref{artilleria} should be considered is a consequence of the whole Extended Euclidean Algorithm one should perform to deal with this problem (see Proposition \ref{sia}). Being in the $j$-th component $\cB_j$ essentially means that the corresponding polynomials in the B\'ezout identity that would solve the RHIP fail to reach their expected degree in at least $j$ steps, and hence one should test up to this number of ``vanishing instances''.

Interestingly, one does not need to use several matrices to deal with the stratification of the set of unattainable points appearing in Theorem \ref{mtt}. Just the  the rank of 
$\M_{\bfu,\bfv,\bfn,k}$ is enough. This is the content of our second main result:
\begin{theorem}\label{mts}
With notations and hypothesis  as above, for $(\bfu, \bfv)\in(\K^l\setminus Z)\times \K^n,$ and  $1\leq j\leq m+1,\,(\bfu,\bfv)$ is such
that the minimal solution of the WHIP associated with this data has degrees bounded by $k-j$ and $n-k+1-j$ respectively if and only if $\dim_\K\left(\mbox{\rm ker}(\M_{\bfu,\bfv,\bfn,k})\right)\geq j.$ 
Moreover, 
$\dim_\K\left(\mbox{\rm ker}(\M_{\bfu,\bfv,\bfn,k})\right)= j$
if and only if the degree bound turns into an equality.
\par
The first situation is given by the following equations
\begin{equation}\label{Dj}
\Delta^{\bfn}_{k-j+2,k-j+2}=\Delta^{\bfn}_{k-j+3,k-j+3}= \dots = \Delta^{\bfn}_{k+j-1,k+j-1} =0
\end{equation}
(if $j=1,$ the above set of equations is empty), while the second one 
is the intersection of \eqref{Dj} with  $\{\Delta^{\bfn}_{k-j+1,k-j+1}\neq 0\} \cup\{\Delta^{\bfn}_{k+j,k+j}\neq 0\}.$
A minimal  solution for the WHIP is given in this region by
\begin{equation}\label{pair}\left\{
\begin{array}{lcl}
\left(\sum_{\ell=0}^{k-j}\Delta^{\bfn}_{k-j+1,\ell+1}x^\ell;\,\sum_{\ell=k-j+1}^{n-2j+2}\Delta^{\bfn}_{k-j+1,\ell+1}x^{\ell-k+j-1} \right) \ & \mbox{in} & \Delta^{\bfn}_{k-j+1,k-j+1}\neq0, \\
\left(\sum_{\ell=0}^{k-j}\Delta^{\bfn}_{k+j-1,\ell+1}x^\ell;\,\sum_{\ell=k+j-1}^{n}\Delta^{\bfn}_{k+j-1,\ell+1}x^{\ell-k-j+1} \right) \ & \mbox{in} & \Delta^{\bfn}_{k+j,k+j}\neq0.
\end{array}\right.
\end{equation}
If $\K$ is algebraically closed, \eqref{Dj} is a rational irreducible complete intersection variety in $(\K^l\setminus Z)\times \K^n$ of codimension $2(j-1)$ in
$(\K^l\setminus Z)\times \K^n.$
\end{theorem}
This result follows from Theorems \ref{hipp}, \ref{oteo}, \ref{mss}, and Proposition \ref{ppear}.
Note that for $j=1,$ both representations in \eqref{pair} are the same up to a constant.
\par
To illustrate Theorem \ref{mts} in our Example \ref{beow} above, as $l=2,\, k=2,\, n=3,\, \bfn=(2,1),$ and $m=1,$ we have
\begin{itemize}
\item $\dim_\K\left(\mbox{\rm ker}(\M_{\bfu,\bfv,(2,1),2})\right)=1 \iff \{\Delta^{(2,1)}_{2,2}\neq0\}\cup\{\Delta^{(2,1)}_{3,3}\neq0\},$
\item $\dim_\K\left(\mbox{\rm ker}(\M_{\bfu,\bfv,(2,1),2})\right)=2 \iff \{\Delta^{(2,1)}_{2,2}=\Delta^{(2,1)}_{3,3}=0\},$ and   $\{\Delta^{(2,1)}_{1,1}\neq0\}\cup\{\Delta^{(2,1)}_{4,4}\neq0\}.$
\end{itemize}
Computing the last two minors, we get:
$$\begin{array}{l}
\Delta^{(2,1)}_{1,1}=\det\left(
\begin{array}{ccc}
-V_{1,0}&-V_{1,0}U_1&-V_{1,0}U_1^2\\
-V_{1,1}&-V_{1,1}U_{1}-V_{1,0}&-V_{1,1}U_1^2-2V_{1,0}U_{1}\\
-V_{2,0}&-V_{2,0}U_2&-V_{2,0}U_2^2
\end{array}
\right)=0, \\ \\
\Delta^{(2,1)}_{4,4}=\det\left(
\begin{array}{ccc}
1&U_1&U_1^2\\
0&1&2U_1\\
1&U_2&U_2^2
\end{array}
\right)=(U_1-U_2)^{2},
\end{array}$$
so $\{\Delta^{(2,1)}_{1,1}\neq0\}\cup\{\Delta^{(2,1)}_{4,4}\neq0\}$ is actually equal to the whole ambient space
$(\K^l\setminus Z)\times \K^n,$ and we have that  $\dim_\K\big(\mbox{\rm ker}(\M_{\bfu,\bfv,(2,1),2})\big)=2$ if and only if
$\{\Delta^{(2,1)}_{2,2}=\Delta^{(2,1)}_{3,3}=0\}.$ In the ``generic'' case, when the dimension is equal to one, a solution \eqref{pair} is given by
the maximal minors of $\M_{\bfu,\bfv,(2,1),2,3},$  i.e.
$$\big(\Delta^{(2,1)}_{2,1}+\Delta^{(2,1)}_{2,2}x;\,\Delta^{(2,1)}_{2,3}+\Delta^{(2,1)}_{2,4}x\big).$$
When $\dim=2,$  \eqref{pair} gives as solutions for the WHIP the nontrivial constant functions
$$\begin{array}{ccl}
\big(\Delta^{(2,1)}_{1,1}; \Delta^{(2,1)}_{1,2}\big)& \mbox{in}\  \Delta^{(2,1)}_{1,1}\neq0,\\
\big(\Delta^{(2,1)}_{3,1}; \Delta^{(2,1)}_{3,4}\big)& \mbox{in}\  \Delta^{(2,1)}_{4,4}\neq0
\end{array}
$$
(The fact that $\Delta^{(2,1)}_{4,4}=\Delta^{(2,1)}_{3,4}$ -see Lemma \ref{deltas}- shows that the second expression above is also not identically zero).

\bigskip
The paper is organized as follows: in Section \ref{s2}
we  show that all the solutions of the WHIP are polynomial multiples of a ``minimal solution,'' in the sense
that all other solutions are polynomial multiples of this one. We show in Theorem \ref{hipp} that a minimal solution solves the
RHIP if and only if its components are coprime polynomials, and  give an algorithm (Algorithm \ref{algho}) to compute the minimal
solution and verify if the RHIP is solvable based on these facts. We end that section by reviewing the Extended Euclidean Algorithm's classical
method to deal with this problem in light of our results in \S \ref{EAR}. All these theoretical results are well known and documented in the literature (see \cite{wuy75, vzGJ13}), we include them here to set the notation within the context of the following statements, and also to set the basis for the algorithms produced therein.
\par In Section \ref{GD} we look at the geometry of the problem. In Theorems \ref{oteo} and \ref{bp} we show that the sets described above are proper union of open irreducible rational varieties  in $\big(\K^l\setminus Z\big)\times(\K^\times)^n.$ Our proof is constructive, so we can develop an algorithm (Algorithm \ref{algop}) to deal with the RHIP without the need of computing any solution of the WHIP.

\par In Section \ref{EPF} we produce the equations which appear in \eqref{Dj} and \eqref{pair}, and complete with the proofs of the main theorems.

\bigskip
\section{minimal solutions}\label{s2}
All along the text we will assume that Char$(\K)=0$ or Char$(\K)\geq\max\{n_1,\ldots, n_l\}.$ In this section, we look at the structure of solutions of the WHIP. We will see that all of them are multiple of a so-called ``minimal solution'',  give an
algorithm for computing this minimal solution, and decide by looking at its decomposition if the RHIP has a solution or not. We will compare our
results with the classical standard procedure for solving the RHIP via the Extended Euclidean Algorithm.

\begin{proposition}\label{uni}
Let $(\bfu,\bfv)\in\left(\K^l\setminus Z\right)\times \K^n.$ For any solution $(A_{\bfu,\bfv,\bfn,k-1}(x), B_{\bfu,\bfv,\bfn,n-k}(x))$ of the WHIP \eqref{whermite} with $B_{\bfu,\bfv,\bfn,n-k}(x)\neq0,$ the rational function
$\frac{A_{\bfu,\bfv,\bfn,k-1}(x)}{B_{\bfu,\bfv,\bfn,n-k}(x)}$ is
unique.
\end{proposition}
\begin{proof}

If both $(A_{\bfu,\bfv,\bfn,k-1}(x),\,B_{\bfu,\bfv,\bfn,n-k}(x))$ and $(\tilde{A}_{\bfu,\bfv,\bfn,k-1}(x),\tilde{B}_{\bfu,\bfv,\bfn,n-k}(x))$ satisfy \eqref{whermite},
the polynomial
$$P(x)=A_{\bfu,\bfv,\bfn,k-1}(x)\tilde{B}_{\bfu,\bfv,\bfn,n-k}(x)-\tilde{A}_{\bfu,\bfv,\bfn,k-1}(x)
B_{\bfu,\bfv,\bfn,n-k}(x)\in\K[x]
$$
has degree bounded by $n-1,$ and its derivatives $
P^{(j)}(x)$
vanishes in the different points $u_i,\,i=1\ldots, l$ for  $j=0,\ldots, n_{i}-1.$
Hence, $P(x)$ must be identically zero. This concludes with the proof of the claim.
\end{proof}

\smallskip
We will study now the structure of the solutions of the WHIP \eqref{whermite}. For $\ell\in\N,$ we denote with $\K[x]_\ell$ the subspace of polynomials in
$\K[x]$ of degree bounded by $\ell.$ For $n, k\in\N,\,1\leq k\leq n,$ and
$(\bfu,\bfv)\in\left(\K^l\setminus Z\right)\times \K^n,$ denote with $\V_{\bfu,\bfv,\bfn,k}\subset\K[x]_{k-1}\oplus\K[x]_{n-k}$ the set defined as
$$\V_{\bfu,\bfv,\bfn,k}=\{(A_{k-1}(x),B_{n-k}(x)) \mbox{ satisfying }\eqref{whermite}
\}.
$$
Note that it is actually the nullspace of the matrix $\M_{\bfu,\bfv,\bfn,k},$ and hence the following claim holds straightforwardly.

\begin{proposition}\label{trump}
$\V_{\bfu,\bfv,\bfn,k}$ is a $\K$-vector subspace of $\K[x]_{k-1}\oplus\K[x]_{n-k}$ of positive dimension.
\end{proposition}

\begin{lemma}\label{genn}
Up to a nonzero constant in $\K,$ there is a unique  $(A^0_{k-1}(x),B^0_{n-k}(x))\in\V_{\bfu,\bfv,\bfn,k}$ such that
$$\deg(A^0_{k-1}(x))=\min\{\deg(A_{k-1}(x)),\,(A_{k-1}(x),B_{n-k}(x))\in\V_{\bfu,\bfv,\bfn,k}\}.$$
\end{lemma}
\begin{proof}
If there are two $(A^0_{k-1}(x),B^0_{n-k}(x)),\,(\tilde{A}^0_{k-1}(x),\tilde{B}^0_{n-k}(x))\in\V_{\bfu,\bfv,\bfn,k}$ such that $A^0_{k-1}(x)$ and
$\tilde{A}^0_{k-1}(x)$ are not proportional, then a nontrivial linear combination of these polynomials will produce an element in $\V_{\bfu,\bfv,\bfn,k}$
of strictly lower degree.
\end{proof}
\begin{definition}
A pair $(A^0_{k-1}(x),B^0_{n-k}(x))\in\V_{\bfu,\bfv,\bfn,k}$ satisfying the hypothesis of Lemma \ref{genn} will be called a {\em minimal element} of
$\V_{\bfu,\bfv,\bfn,k}.$
\end{definition}
The minimal element is unique if we require $A^0_{k-1}(x)$ to be monic.
The following result gives some light on the structure of the $\K$-vector space $\V_{\bfu,\bfv,\bfn,k}.$
\begin{proposition}\label{2.5}
Any element in $\V_{\bfu,\bfv,\bfn,k}$ is a polynomial multiple of a minimal element of this space.
\end{proposition}
\begin{proof}
Clearly any polynomial multiple of an element of $\V_{\bfu,\bfv,\bfn,k}$ 
(as long as the degrees of each of the two components do not go above $k-1$ and $n-k$ respectively) is an element
of $\V_{\bfu,\bfv,\bfn,k}.$ 

Let $(A_{k-1}(x),B_{n-k}(x))\in\V_{\bfu,\bfv,\bfn,k}$ with $B_{n-k}(x)\neq 0$ and $A^{00}_{k-1}(x),B^{00}_{n-k}(x)\in\K[x]$ be such that $\gcd(A^{00}_{k-1}(x),B^{00}_{n-k}(x))=1$ and
$\frac{A_{k-1}(x)}{B_{n-k}(x)}=\frac{A^{00}_{k-1}(x)}{B^{00}_{n-k}(x)}.$ 

If $(A^{00}_{k-1}(x), B^{00}_{n-k}(x))\in \V_{\bfu,\bfv,\bfn,k}$ then  $(A^{00}_{k-1}(x), B^{00}_{n-k}(x))$ is a minimal element in  $\V_{\bfu,\bfv,\bfn,k}$ and the claim follows from Proposition \ref{uni}. Otherwise, let 
$I$ be the set of those $i\in \{1,\ldots, l\}$ such that  there is $j=0,\ldots, n_i-1$ such that $(A^{00}_{k-1}(x), B^{00}_{n-k}(x))$ does not satisfy the identity 
$A_{\bfu,\bfv,\bfn,k-1}^{(j)}(u_i)=\sum_{t=0}^{j}(j)_t v_{i,t}\,B^{(j-t)}_{\bfu,\bfv,\bfn,n-k}(u_i)$ coming from \eqref{whermite}.  For $i\in I$, let
 $j_i= \min\{j\,|\,(A^{00}_{k-1})^{(j)}(u_i)\neq \sum_{t=0}^{j}(j)_tv_{i,j}(B^{00}_{n-k})^{(j-t)}(u_i)\}.$ Thanks to Proposition \ref{uni}, for any $(A_{k-1}(x),B_{n-k}(x))\in\V_{\bfu,\bfv,\bfn,k}$  there exists $C(x)\in\K[x]$ such that
$$A_{k-1}(x)=C(x)A^{00}_{k-1}(x),\, B_{n-k}(x)=C(x)B^{00}_{n-k}(x).$$
We claim that $C^{(j)}(u_i)=0$ for all $i\in I$ and $j=0,\ldots,n_i-j_i$. Then, from the claim, we deduce that a minimal element in  $\V_{\bfu,\bfv,\bfn,k}$ is $(A^0_{k-1}(x),B^0_{n-k}(x))= (\tilde{C}(x)A^{00}_{k-1}(x),\tilde{C}(x)B^{00}_{n-k}(x))$ with $\tilde{C}(x)=\prod_{i\in I}(x-u_i)^{n_i-j_i}\,| C(x)$.

To prove the claim we use induction on $j$. For $j=0,$ we have
$$A^{(j_{i})}_{k-1}(u_{i})=\sum_{t=0}^{j_{i}}{j_{i}\choose t}C^{(t)}(u_{i})(A^{00}_{k-1})^{(j_{i}-t)}(u_{i})=\sum_{s=0}^{j_{i}} (j_i)_sv_{i,s}(B_{n-k})^{(j_{i}-s)}(u_{i})$$
$$=\sum_{s=0}^{j_{i}}(j_i)_sv_{i,s}\sum_{t=0}^{j_{i}-s}{j_{i}-s\choose t}C^{(t)}(u_{i})(B^{00}_{n-k})^{(j_{i}-s-t)}(u_{i})=$$
$$=\sum_{t=0}^{j_{i}}{j_{i} \choose t}C^{(t)}(u_{i}) \sum_{s=0}^{j_{i}-t}(j_i-t)_sv_{i,s}(B^{00}_{n-k})^{(j_{i}-t-s)}(u_{i}) ,
$$
and we deduce from the definition of $j_{i}$ that $C(u_{i})=0$. Assume now  $1\leq j \leq n_i-j_i-1.$ Hence,
 \begin{equation}\label{tiki}
 \begin{array}{cl}
 A^{(j_{i}+j)}_{k-1}(u_{i})&=\sum_{t=0}^{j_{i}+j}{j_{i}+j\choose t}C^{(t)}(u_{i})(A^{00}_{k-1})^{(j_{i}+j-t)}(u_{i}) \\
 & =\sum_{s=0}^{j_{i}+j}(j_{i}+j)_s
 v_{i,s}(B_{n-k})^{(j_{i}+j-s)}(u_{i}) \\
&=\sum_{s=0}^{j_{i}+j}(j_{i}+j)_s v_{i,s}\sum_{t=0}^{j_{i}+j-s}{j_{i}+j-s\choose t}C^{(t)}(u_{i})(B^{00}_{n-k})^{(j_{i}+j-s-t)}(u_{i})
\\ &=\sum_{t=0}^{j_{i}+j}{j_{i}+j \choose t}C^{(t)}(u_{i}) \sum_{s=0}^{j_{i}+j-t}(j_{i}-t)_sv_{i,s}(B^{00}_{n-k})^{(j_{i}+j-t-s)}(u_{i}).
\end{array}
\end{equation}
By the inductive hypothesis, we have that $C^{(t)}(u_i)=0$ for $0\leq t\leq j-1,$ and hence \eqref{tiki} above becomes 
$$ \sum_{t=j}^{j_{i}+j}{j_{i}+j\choose t}C^{(t)}(u_{i})(A^{00}_{k-1})^{(j_{i}+j-t)}(u_{i}) =\sum_{t=j}^{j_{i}+j}{j_{i}+j \choose t}C^{(t)}(u_{i}) \sum_{s=0}^{j_{i}+j-t}(j_{i}-t)_sv_{i,s}(B^{00}_{n-k})^{(j_{i}+j-t-s)}(u_{i}).
$$
Due to the conditions imposed on the characteristic of $\K,$ we have  ${j_{i}+j\choose j}\neq0,$ and hence we deduce  that  $C^ {(j)}(u_{i})=0.$
\end{proof}
\smallskip
All the previous claims imply the following
\begin{theorem}\label{hipp}
Let $\K$ be a field  with Char$(\K)=0$ or Char$(\K)\geq\max\{n_1,\ldots, n_l\}.$
For a given $(\bfu,\bfv)\in\left(\K^l\setminus Z\right)\times \K^n,$ let $d_A$ (resp. $d_B$) denote the degree
of $A^0_{k-1}(x)$ (resp. $B^0_{n-k}(x)$) for a minimal element $(A^0_{k-1}(x),B^0_{n-k}(x))\in\V_{\bfu,\bfv,\bfn,k},$  and set
$s_0:=\min\{k-1-d_A,\,n-k-d_B\}.$ Then
\begin{itemize}
\item $\V_{\bfu,\bfv,\bfn,k}=(A^0_{k-1}(x),B^0_{n-k}(x))\cdot\K[x]_{s_0}.$
\item $\dim_\K(\V_{\bfu,\bfv,\bfn,k})=s_0+1.$
\item The RHIP is solvable if and only if $\gcd(A^0_{k-1}(x),B^0_{n-k}(x))=1.$
\end{itemize}
\end{theorem}
\begin{proof}
The first two claims follow straightforwardly from Proposition \ref{2.5}. For the last, following the notation of the proof of
Proposition \ref{2.5}, note that $(A^0_{k-1}(x),B^0_{n-k}(x))$ solves the RHIP if and only if the set of indexes $I$ is the empty set, which is equivalent
to  $\gcd(A^0_{k-1}(x),B^0_{n-k}(x))=1.$
\end{proof}
\smallskip
Thanks to Theorem \ref{hipp} we can produce the following algorithm.
\begin{algorithm}\label{algho}$^{}$
\par\noindent\underline{Input:} $(\bfu,\bfv)\in\big(\K^l\setminus Z)\times \K^n,$ $\K$ being a field with Char$(\K)=0$ or Char$(\K)\geq\max\{n_1,\ldots, n_l\}.$
\par\noindent\underline{Output:}  A reduced solution of the RHIP associated to $(\bfu,\bfv),$ or message that it does not have a solution.
\begin{enumerate}
\item Compute a nontrivial element of the kernel of the matrix  $\M_{\bfu,\bfv,\bfn,k}$ associated to the linear system \eqref{whermite}.
\item Extract polynomials $A_{k-1}(x)$ and $B_{n-k}(x)$ from the coordinates of the element computed in (1).
\item Remove the common factors of these two polynomials. Denote the reduced polynomials by $A^{00}_{k-1}(x)$ and $B^{00}_{n-k}(x)$.
\item if $(A^{00}_{k-1}(x),B^{00}_{n-k}(x))\in\mbox{\rm ker}(\M_{\bfu,\bfv,\bfn,k}),$ then return
$(A^{00}_{k-1}(x),B^{00}_{n-k}(x)),$ otherwise return ``no solution''.
\end{enumerate}
\end{algorithm}

\begin{remark}
If instead of ``a reduced solution'' one simply looks for ``a solution'', step (3) in Algorithm \ref{algho} can be
replaced with ``remove the common factors that these two polynomials have of the type $(x-u_i),\,i=1,\ldots, l$''.
\end{remark}

\subsection{Euclidean Algorithm Revisited}\label{EAR}
We review here the standard computational method used to deal with the RHIP, and compare it with our explorations. A good reference for this part of
the text is \cite[\S 5.8]{vzGJ13}. It turns out that the key connection to deal efficiently with this problem is the classical Euclidean Algorithm as Proposition \ref{sia} below show. Recall that for polynomials $F(x), G(x)\in\K[x],$   the \textit{Extended Euclidean Algorithm} consists of a finite matrix of $5$ columns $(i, q_i, r_i, s_i, t_i)$ such that
\begin{itemize}
\item the zeroth row is $(0, 0, F(x), 1, 0),$
\item the first row is $(1,Q_1(x),G(x),0,1),\, Q_1(x)$ being the Euclidean quotient between $F(x)$ and $G(x)$,
\item for $i\geq 2$, the $i$-th row is $(i, Q_i(x), R_i(x),S_i(x),T_i(x))$, where $R_i(x)$ is the remainder of the Euclidean division between $R_{i-2}(x)$ and $R_{i-1}(x),$ and $S_i(x),\,T_i(x)$ are the coefficients of the B\'ezout identity satisfying $$S_i(x)F(x)+T_i(x)G(x)=R_i(x).$$
\end{itemize}
Note that for $i\geq2,$ we have $\deg(R_i(x))<\deg(R_{i-1}(x)).$ In addition, we have that (\cite[Lemmas 3.8 \& 3.10]{vzGJ13})
\begin{equation}\label{bbund}
\begin{array}{ccl}
\deg(S_i(x))&\leq &\deg(G)-\deg(R_{i-1}(x))\\
\deg(T_i(x))&= &\deg(F)-\deg(R_{i-1}(x)),\\
\gcd(R_i(x),T_i(x))&=&\gcd(F(x),T_i(x)) \, \forall i\geq0.
\end{array}
\end{equation}
One of the interesting features of the EEA is that it produces ``short'' (in degree) B\'ezout identities in the following sense:
\begin{lemma}\label{ax}(\cite[Lemma 5.15]{vzGJ13})
Let $R(x),\,S(x),\,T(x)$ be such that $S(x)F(x)+T(x)G(x)=R(x)$ with $\deg(R(x))+\deg(T(x))<\deg(F(x)),$ and $\ell\in\N$ the index that
$$\deg(R_\ell(x))\leq \deg(R(x))<\deg(R_{\ell-1}(x)).$$
Then, there exists $C(x)\in\K[x]$ such that $$R(x)=C(x)R_\ell(x),\,S(x)=C(x)S_\ell(x),\,T(x)=C(x)T_\ell(x).$$
\end{lemma}

We conclude this section by recovering the fundamental result in \cite{vzGJ13} which essentially states that one can solve the RHIP by looking at a specific row in the Extended Euclidean Algorithm.

\begin{proposition}\label{sia}(\cite[Exercise 5.42]{vzGJ13}])
With notation as above, for a given data $(\bfu,\bfv)\in\big(\K^l\setminus Z\big)\times\K^n,$ set
$F(x)=\prod_{i=1}^{l}F_{i}(x)=\prod_{i=1}^l(x-u_i)^{n_{i}},$ and let $G(x)\in\K[x]$ be the unique interpolating polynomial
of degree less than $n$ which satisfies $G^{(j)}_{i}(u_i)=j!\,v_{i,j},\,i=0,\ldots, l,\,j=1,\ldots,n_{i}-1.$
Denote with $\ell$ the minimal index such that
$\deg(R_\ell(x))\leq k-1.$   The RHIP is solvable if and only if $\gcd(R_\ell(x),T_\ell(x))=1.$ If this is the case, then
the pair $(R_\ell(x),T_\ell(x))$ is a solution of it.
\end{proposition}
\begin{proof}

Given $A(x)$, $B(x) \in  \K[x]$ with $\deg(A(x))\leq k-1$ and $\deg(B(x))\leq n-k$ we have that
the pair $(A(x),B(x))$ gives a solution for the WHIP iff \linebreak $A^{(j)}(u_i)=(G(x)B(x))^{(j)}(u_i)$ for $ i=1,\ldots, l$ and $j=0,\ldots ,n_i -1$, equivalently if
$(A(x)-G(x)B(x))^{(j)}(u_i)=0$ for $ i=1,\ldots, l$ and  $j=0,\ldots ,n_i -1$; that is, iff $F(x)$ divides $A(x)- G(x)B(x)$ and so, iff $A(x)=S(x)F(x)+B(x)G(x) $ for some  $S(x)\in \K[x]$.

Let $(A^0_{k-1}(x),B^0_{n-k}(x))$ be a  minimal element in $\V_{\bfu,\bfv,\bfn,k}$. Then, since the pair $(R_{\ell}(x), T_{\ell}(x))$ gives a solution for the WHIP, we have, by Proposition \ref{2.5}, that
$(R_{\ell}(x), T_{\ell}(x))=(C(x)A^0_{k-1}(x), C(x)B^0_{n-k}(x))$ for some $C(x)\in \K[x]$. In particular, $\deg(B^0_{n-k}(x)) \leq \deg( T_{\ell}(x))$.

On the other hand, since $A^0_{k-1}(x)= S(x)F(x)+ B^0_{n-k}(x)G(x)$ for some $S(x)\in \K[x]$, $\deg(A^0_{k-1}(x)) + \deg(B^0_{n-k}(x))<\deg(F(x))$ and $\deg (A^0_{k-1}(x)) \leq \deg(R_{\ell}(x))$ we have, by Lemma \ref{ax}, that there exist $\ell' \geq\ell$ and $D(x) \in \K[x]$ such that
$(A^0_{k-1}(x),B^0_{n-k}(x) )= (D(x)R_{\ell'}(x),D(x)T_{\ell'}(x))$. In particular, we have that \linebreak $\deg( T_{\ell'}(x))\leq \deg(B^0_{n-k}(x))$.

Thus, $\deg( T_{\ell'}(x))\leq \deg(B^0_{n-k}(x))\leq \deg( T_{\ell}(x))$ and, $\deg( T_{\ell'}(x))\geq \deg( T_{\ell}(x))$ by \eqref{bbund}. So, $\ell =\ell'$ and  the pairs $(A^0_{k-1}(x),B^0_{n-k}(x))$  and $(R_{\ell}(x), T_{\ell}(x))$ are unique up to a constant in $\K$. The statement now follows from Theorem \ref{hipp}.

\end{proof}

\bigskip
\section{Geometric Description}\label{GD}
In this section we will study the geometry of the RHIP. We start by showing that all the matrices which are of interest in our problem have ``generic'' maximal rank if the field is large enough and its characteristic is zero or also large enough. Recall that we are always under the assumption that Char$(\K)=0$ or Char$(\K)\geq\max\{n_1,\ldots, n_l\}.$
\begin{proposition}\label{ufi}
Let $\alpha,\,\beta,\,n_{1},\ldots,n_{l}\in\Z_{\geq0}.$ If $\K$ is a field such that  $\#(\K)\geq l.$   Then, there are elements
$(\bfu,\bfv)\in(\K^l\setminus Z)\times\K^n$ such that $\M_{\alpha,\beta,{\bf n}}(\bfu,\bfv)$ has maximal rank.
\end{proposition}
\begin{proof}
For arbitrary $u_1,\ldots, u_l\in\K,$ we specialize 
\begin{equation}\label{paramm}
U_i\mapsto u_i,\,V_{i,j}\mapsto  -{\alpha+1 \choose j}u_i^{\alpha-j+1},\,1\leq i\leq l,\ 0\leq j\leq n_i-1
\end{equation} in \eqref{symbM}
and the matrix become a confluent Vandermonde one, which is known to have maximal rank  for any  $\bfu\in\K^l\setminus Z.$
\end{proof}
\smallskip

We would also need to show that for generic solutions of the RHIP, the denominator $B(x)$ is such that $B(u_i)\neq 0,\,i=1,\ldots, l.$

\begin{lemma}\label{build}
Let $n_{1},\ldots,n_{l}\in\Z_{\geq0},$ and $x, U_1,\ldots, U_l$  indeterminates over $\K.$ The polynomial
\begin{equation}\label{ino}
(x-U_1)^{n_1}\ldots (x-U_l)^{n_l}= \sum_{j=0}^{n}c_j(U_1,\ldots, U_l)x^j
\end{equation}
satisfies $c_j(U_1,\ldots, U_l)\neq0$ for all $j=0,\ldots, l.$
\end{lemma}
\begin{proof}
We expand the left hand side of \eqref{ino} to get
$$
\sum_{0\leq \beta_i\leq n_i,\,1\leq i\leq l}(-1)^{n-\beta_1-\ldots-\beta_l}{n_1\choose\beta_1}\ldots {n_l\choose\beta_l}U_1^{\beta_1}\ldots U_l^{\beta_l}x^{\beta_1+\ldots+\beta_l}.
$$ From here, we deduce straightforwardly that
$$c_j(U_1,\ldots, U_l)=(-1)^{n-j}
\sum_{0\leq \beta_i\leq n_i,\,\sum \beta_i=j}{n_1\choose\beta_1}\ldots {n_l\choose\beta_l}U_1^{\beta_1}\ldots U_l^{\beta_l},
$$
is a nonzero polynomial, which concludes with the proof of the claim.
\end{proof}
\smallskip

\begin{proposition}\label{B(ui)}
Let $\alpha,\beta,n_1,\ldots, n_l\in\N$ be such that $\alpha+\beta+1\geq n_1+\ldots+n_l,$ and $\K$ a field of large cardinality.  Then, there exists $(\bfu,\bfv)\in(\K^l\setminus Z)\times\K^n$ such that the RHIP has a solution (i.e. the denominator $B(t)$ satisfies $B(u_i)\neq0, \,\forall i=1,\ldots, n$).
\end{proposition}
\begin{proof} We can assume w.l.o.g. that $\alpha+\beta+1=n_1+\ldots+n_l$ and, similarly to  \eqref{paramm}, do the substitution
$V_{i,j}\mapsto  -{\alpha+1 \choose j}U_i^{\alpha-j+1},\,1\leq i\leq l,\ 0\leq j\leq n_i-1.
$
 We do the substituion \eqref{paramm}, and have a matrix of maximal rank of confluent Vandermonde type, We call it $V_{\alpha,\beta}(\bfU).$ Note that if we add the following row to  $V_{\alpha,\beta}(\bfU):\, (1,\, x\,\ldots,x^{\alpha+\beta+1}),$ we then get a square matrix whose determinant is equal to $w\,\prod_{i=1}^l(x-U_i)^{n_i},$ with $w\in\K\setminus\{0\}.$ Write the latter as $A(x)+x^{\alpha+1}B(x),$ and note that the pair $(A(x),B(x))$ solves the WHIP for $k=\alpha+1.$ We must have $(A(x),B(x))\neq(0,0)$ as the coefficients of these two polynomials are the maximal minors of $V_{\alpha,\beta}(\bfU)$ which -thanks to Proposition \ref{ufi}- is of maximal rank.
 \par Due to the fact that $\K$ is large enough, the proof will be completed if we show that $B(U_i)\neq0$ for all $i=1,\ldots, n.$ Suppose that for some $i_0$ we have $B(U_{i_0})=0.$ Then we must have that $A(U_{i_0})=0,$ and hence  we deduce that
 \begin{equation}\label{quii}
 \frac{w\prod_{i=1}^l(x-U_i)^{n_i}}{x-U_{i_0}}=\tilde{A}(x)+x^{\alpha+1}\tilde{B}(x),
 \end{equation}
 with $\deg\big(\tilde{A}(x)\big)\leq \alpha-1.$ This shows that the right hand side of \eqref{quii} has the coefficient of $x^\alpha$ equal to zero, which is a contradiction with  Lemma \ref{build} applied to the data $(n_1,\ldots, n_{i_0}-1,\ldots, n_l),$  and completes the proof of the Proposition. 
\end{proof}
\smallskip

Recall that $m=\min\{k-1,n-k\}.$ We will fix $j\in\{0,\ldots, m\}.$ Let $\cF^j_{k,\bfn}\subset\K[\bfU,\bfV]$ be the ideal
of $(n-2j)\times(n-2j)$ minors of $\M_{k-j-1,n-k-j,\bfn}(\bfU, \bfV).$ For convenience, we set $\cF^{-1}_{k,\bfn}:=\langle 0\rangle.$ 
As $\M_{k-(j+1)-1,n-k-(j+1),\bfn}(\bfU,\bfV)$ is a submatrix of $\M_{k-j-1,n-k-j,\bfn}(\bfU,\bfV),$ the following
claim holds by performing a suitable Laplace expansion on the minors of the latter matrix.
\begin{lemma}\label{filt}
For $j=0,\ldots, m-1,$
$\cF^{j}_{k,\bfn}\subset\cF^{j+1}_{k,\bfn}.$
\end{lemma}
From Theorem \ref{hipp}, we deduce
that $1\leq\dim_\K(\V_{\bfu,\bfv,\bfn,k})\leq m+1.$ We will look at a geometric description of this stratification. For $0\leq j\leq m,$
consider the algebraic set $V(\cF^{j}_{k,\bfn})\subset\big(\K^l\setminus Z)\times \K^n$ defined by the vanishing of
the elements of $\cF^{j}_{k,\bfn}.$ No further assumptions on the field $\K$ are required for the following statements.

\begin{theorem}\label{bg}
For $j=1,\ldots,m+1,\, (\bfu,\bfv)$ is such that $\dim_\K(\V_{\bfu,\bfv,\bfn,k})=j$ if and only if
$(\bfu,\bfv)\in V(\cF^{j-2}_{k,\bfn})\setminus V(\cF^{j-1}_{k,\bfn}).$
\end{theorem}
\begin{proof}
Recall from Proposition \ref{trump} that $\V_{\bfu,\bfv,\bfn,k}$ can be identified
with $\mbox{\rm ker}(\M_{\bfu,\bfv,\bfn,k}).$ Its dimension equal to one if and only if this matrix has maximal rank, which is equivalent
to $(\bfu,\bfv)\notin V(\cF^0_{k,\bfn}).$ This proves the claim for $j=1.$
\par
Suppose now that $\dim_\K(\V_{\bfu,\bfv,\bfn,k})=j+1,$ with $j>0.$ It is easy to see that the $(n-2j+1)$-th dimensional vector of the
coefficients of the minimal solution $(A^0_{k-1}(x),B^0_{n-k}(x))$ actually belongs to $\mbox{\rm ker}(\M_{k-j-1,n-k-j,\bfn}(\bfu,\bfv)),$ and
hence the coefficients of both $(A^0_{k-1}(x),B^0_{n-k}(x))$ and $(xA^0_{k-1}(x),xB^0_{n-k}(x))$ belong to the kernel
of $\M_{k-(j-1)-1,n-k-(j-1),\bfn}(\bfu,\bfv),$ which implies that the dimension of the kernel of this matrix is larger than one, and
hence we deduce that $(\bfu,\bfv)\in  V(\cF^{j-1}_{k,\bfn}).$ If in addition we had  $(\bfu,\bfv)\in V(\cF^{j}_{k,\bfn}),$ then there
would be another element in  $\mbox{\rm ker}(\M_{k-j-1,n-k-j,\bfn}(\bfu,\bfv))$ linearly independent with the vector of
coefficients of $(A^0_{k-1}(x),B^0_{n-k}(x)),$ and encoding a pair of polynomials of degrees bounded by $k-j-1$ and $n-k-j$ respectively,
which also belongs to $\V_{\bfu,\bfv,\bfn,k}.$ As we have that either $\deg(A^0_{k-1}(x))=k-j-1$ or $\deg(B^0_{n-k}(x))=n-k-j,$ we
conclude that this other element would then be a scalar  multiple of $(A^0_{k-1}(x),B^0_{n-k}(x))$, which is a contradiction.

\par Reciprocally, if $(\bfu,\bfv)\in V(\cF^{j-1}_{k,\bfn}),$ then  $$\dim_\K(\mbox{\rm ker}(\M_{k-(j-1)-1,n-k-(j-1),\bfn}(\bfu,\bfv))\geq2.$$
As all the polynomials $(A(x),B(x))$ coming from the coordinates of elements in this kernel belong to $\V_{\bfu,\bfv,\bfn,k},$ we deduce
that the minimal solution of this space is of the form $(A^0_{k-1}(x), B^0_{n-k}(x))$ of degrees bounded by $k-j-1$ and $n-k-j$ respectively. Note that this vector also produces an element in $\mbox{\rm ker}(\M_{k-j-1,n-k-j,\bfn}(\bfu,\bfv)).$ If in addition $(\bfu,\bfv)\notin V(\cF^j_{k,\bfn}),$ then the coordinates of $(A^0_{k-1}(x), B^0_{n-k}(x))$ must be the only  -up to a constant- element of this kernel, which implies straightforwardly that either $k-j-1=\deg(A^0_{k-1}(x)),$ or $n-k-j=\deg(B^0_{n-k}(x)).$ The fact that $\dim_\K(\V_{\bfu,\bfv,\bfn,k})=j+1$ follows now straightforwardly from Theorem \ref{hipp}.
\end{proof}
\smallskip
From the proof of Theorem \ref{bg}, we deduce immediately the following characterizations.
\begin{corollary}\label{fp}
For $j=0,\ldots, m, V(\cF^{j-1}_{k,\bfn})$ is equal to set of all $(\bfu,\bfv)\in(\K^l\setminus Z)\times\K^n$ such that the minimal
solution of $\V_{\bfu,\bfv,\bfn,k}$ has degrees bounded by $k-1-j$ and $n-k-j$ respectively.
\end{corollary}
\begin{corollary}\label{ox}
For $j=0,\ldots, m,$
$$(\bfu,\bfv)\in V(\cF^{j-1}_{k,\bfn})\setminus V(\cF^{j}_{k,\bfn})\iff\dim_\K\left(\mbox{\rm ker}\big(\M_{k-j-1,n-k-j,\bfn}(\bfu,\bfv)\big)\right)=1.$$
\end{corollary}

\medskip
To study irreducibility and rationality of our objects of interest, we will need to use the language of  and tools from Algebraic Geometry. So we will
assume for the rest of this section that our field $\K$ is contained in an algebraically closed field $K$, which is where our
geometric statements will take place. Given $\alpha,\,\beta\in\Z_{\geq0},$ $\bfn=(n_{1},\ldots,n_{l})\in \Z_{>0}$ let
$W_{\bfn,\alpha,\beta}\subset\left(K^l\setminus Z\right)\times K^n\times\P^{\alpha+\beta+1}(K)$ be the incidence variety given by
\begin{equation}\label{incid}
\begin{array}{c}
(\bfu,\bfv,a_0:\ldots:a_{\alpha}:b_0:\ldots :b_{\beta})\in W_{\bfn,\alpha,\beta}
\\
\iff \\
A^{(j)}(u_i)=\sum_{t=0}^{j}(j)_t\, v_{i,t}\,B^{(j-t)}(u_i),\,i=1,\ldots, l,\,j=0,\ldots,n_{i}-1;
\end{array}
\end{equation}
where $A(x)$ and $B(x)$ are defined as
\begin{equation}\label{ND}
A(x)=\sum_{\ell=0}^\alpha a_\ell x^\ell,\,B(x)=\sum_{\ell=0}^\beta b_\ell x^\ell.
\end{equation}

\begin{theorem}\label{osu}
$W_{\bfn,\alpha,\beta}$ is an irreducible variety of dimension $l+\alpha+\beta+1,$ defined over $\K$.
\end{theorem}

\begin{proof}
The proof will be done by induction on $n=|\bfn|,$ the initial case being obvious, as $W_{1,\alpha,\beta}$ is an irreducible hypersurface defined by one polynomial with coefficients in $\K\subset K.$ Suppose then $n>1,$ Let $A_i,\,B_j,\, U_k,$ and $V_{k,t}$ with $0\leq i\leq\alpha,\,0\leq j\leq\beta, 1\leq k\leq l,\, 0\leq t\leq n_k-1,$ be distinct indeterminates over $K.$ We will work in the localized ring
$${\rm R}:=K[U_1,\ldots, U_l, V_{1,0},\ldots, V_{l,n_l-1}, A_0,\ldots, A_{\alpha},B_0,\ldots, B_{\beta}]_{\prod_{1\leq i<j\leq n}(U_i-U_j)}.$$
Set ${\bf A}(x):=\sum_{\ell=0}^\alpha A_\ell x^\ell,$ and ${\bf B}(x):=\sum_{\ell=0}^\beta B_\ell x^\ell,$ which are elements in $R[x].$ 
Denote with $I$ the ideal of ${\rm R}$ defined by  
$${\bf A}^{(j)}(U_i)-\sum_{t=0}^{j}(j)_t\, V_{i,t}\,{\bf B}^{(j-t)}(U_i),\,i=1,\ldots, l,\,j=0,\ldots,n_{i}-1.
$$
Note that $W_{\bfn,\alpha,\beta}=V(I).$ Let $J\subset {\rm R}$ be the kernel of the  map
\begin{equation}\label{mapa}
\begin{array}{ccl}
{\rm R}&\to&K(U_1,\ldots, U_l,A_0,\ldots, A_{\alpha},B_0,\ldots, B_{\beta})\\
A_j&\mapsto &A_j,\,j=0,\ldots, \alpha\\
B_k&\mapsto &B_k,\,k=0,\ldots, \beta\\
U_i&\mapsto& U_i,\,i=1,\ldots, l\\
V_{i,t}&\mapsto &\frac{1}{t!}\big(\frac{\bf A}{\bf B}\big)^{(t)}(U_i),\,t=0,\ldots, n_i-1.
\end{array}
\end{equation}

From its definition, we deduce straightforwardly that  $J$ is a prime ideal, and homogeneous in the variables $A_i,\,B_j.$ It is also clear that  $V(J)\subset\left(K^l\setminus Z\right)\times K^n\times\P^{\alpha+\beta+1}(K)$ is irreducible of dimension $l+\alpha+\beta+1$. In addition, we have
$$ I\subset J=I:\langle\prod_{i=1}^l \bfB(U_i)\rangle^N
$$
for a suitable $N\in\N,$
and hence we deduce that
$ V(I:\langle\prod_{i=1}^l\bfB(U_i)\rangle^N)=V(J)\subset V(I).$ In particular, the dimension of $V(I)$ is at least $l+\alpha+\beta+1.$
As
$$V(I:\langle\prod_{i=1}^l\bfB(U_i)\rangle^N)=\overline{V(I)\setminus V(\prod_{i=1}^l\bfB(U_i))}=
\overline{V(I)\setminus \cup_{i=1}^lV(\bfB(U_i)})=\overline{\cap_{i=1}^lV(I)\setminus V(\bfB(U_i))}$$
(the first equality follows because $K$ is algebraically closed, see for instance \cite[Section 4, Theorem 7]{CLO15}), the claim will hold if we show that the dimension of $V(I)\cap V(\bfB(U_i))$ is strictly smaller than $l+\alpha+\beta+1$ for some $i=1,\ldots, l.$ We will actually show that this will happen for all $i.$ Indeed, suppose w.l.o.g. that $i=1,$ and set $I_1:=I+\langle \bfB(U_1)\rangle.$ We clearly have that  $V(I_1)=V(I)\cap V(\bfB(U_1)),$ and moreover
$(\bfu,\bfv,a_0:\ldots:a_{\alpha}:b_0:\ldots:b_{\beta})\in V(I_1)$ if and only if $(x-u_1)$ divides both $A(x)$ and $B(x),$ and by setting $A(x)=(x-u_1)\tilde{A}(x)$ and $B(x)=(x-u_1)\tilde{B}(x),$ we have that
\begin{equation}\label{oski}
\tilde{A}^{(j)}(u_i)=\sum_{t=0}^{j}(j)_t\, V_{i,t}\,\tilde{B}^{(j-t)}(u_i),\,i=2,\ldots, l,\,j=0,\ldots,n_{i}-1,  \end{equation} 
and, if  $n_1\geq2,$
\begin{equation}\label{dtt}
{\tilde{A}}^{(j)}(u_1)=\sum_{t=0}^{j}(j)_t\, V_{1,t}\,\tilde{ B}^{(j-t)}(u_1),\,j=0,\ldots,n_{1}-2.
\end{equation}
If either $\alpha$ or $\beta$ is equal to zero, these conditions imply  $A(x)=B(x)=0,$ and hence $V(I_1)=\emptyset,$ so the claim follows for this case straightforwardly. Otherwise, we have
\begin{equation}\label{v}
V(I_1) \simeq \left( K^2\times W_{n_2,\ldots, n_l,\alpha-1,\beta-1}\right)\setminus (Z\times K^n)
\end{equation}
if  $n_1=1$ (as $u_1$ and $v_{1,0}$ can be chosen arbitrarily), or  
\begin{equation}\label{vi}
V(I_1)\simeq \left( K\times W_{n_1-1,n_2,\ldots, n_l,\alpha-1,\beta-1}\right)\setminus (Z\times K^n)
\end{equation} 
as the only ``choice'' here is given by $v_{1,n_1-1}.$

By the Induction Hypothesis, in both cases we have that  the dimension of these varieties is $l+\alpha+\beta,$
which is strictly smaller than $\dim(V(I)).$
This completes the proof of the Theorem.
\end{proof}

Let $\pi:W_{{\bfn},\alpha,\beta}\to\left(K^l\setminus Z\right)\times K^n$ be the projection onto the first factor. By the Elimination Theorem, we have that $\pi(W_{{\bfn},\alpha,\beta})$ is an irreducible variety, also defined over $\K$.

\begin{theorem}\label{pret}
With notation as above, $\pi(W_{\bfn,\alpha,\beta})$ is a rational variety of dimension $l+\min\{n, \alpha+\beta+1\},$ defined over $\K$.
\end{theorem}
\begin{proof}
Thanks to Theorem \ref{osu}, we have 
$$\dim\big(\pi(W_{\bfn,\alpha,\beta})\big)\leq l+\alpha+\beta+1=\dim(W_{\bfn,\alpha,\beta}).$$  
Let $\M_{\alpha,\beta,\bfn}(\bfU,\bfV)$ be the  $n\times(\alpha+\beta+2)$ matrix defined in \eqref{symbMM}.
It is straightforward to check that
\begin{equation}\label{axx}
(\bfu, \bfv)\in\pi(W_{\bfn,\alpha,\beta})\iff \dim_K\big(\mbox{\rm ker}(\M_{\alpha,\beta,\bfn}(\bfu,\bfv))\big)\geq1.
\end{equation}
If $n<\alpha+\beta+2$ the right hand side of \eqref{axx}  holds straightforwardly, hence  $\pi(W_{\bfn,\alpha,\beta})$ is the whole space $\big(K^l\setminus Z\big)\times K^n.$ So, the claim holds for this case. Otherwise, we will have that  the rank of $\M_{\alpha,\beta, \bfn}(\bfu, \bfv)$ will be less than or equal to $\alpha+\beta+1$ for those $(\bfu,\bfv)$ satisfying \eqref{axx}. 
Let $\overline{\bfn}:=(\overline{n}_1,\ldots, \overline{n}_l)\in\N^l$ be such that $\overline{n}_i\leq n_i$ for $i=1,\ldots, l,$ and $\overline{n}:=\sum_{i=1}^l\overline{n}_i=\alpha+\beta+1<n.$

Thanks to Proposition \ref{ufi} applied to the field $K,$ the matrix $\M_{\alpha,\beta,\overline{\bfn}}(\bfU,\bfV)$ generically has maximal rank. Let $\cU_0\subset\big(K^l\setminus Z\big)\times K^{\alpha+\beta+1}$ be the nonempty open set consisting of all the data $(\bfu,\bfv)$ such that $\M_{\alpha,\beta,\overline\bfn}(\bfu,\bfv)$  has maximal rank. We set now the following rational map
\begin{equation}\label{mmap}
\begin{array}{ccc}
\cU_0& \dashrightarrow &\pi(W_{\bfn,\alpha,\beta}) \\
(\bfu,\bfv)&\mapsto& \big(\bfu,\bfv, \frac{1}{t!}\big(\frac{A}{B}\big)^{(t)}(u_i),\,1\leq i\leq l,\,\overline{n}_i\leq t<n_i\big),
\end{array}
\end{equation}
where the coefficients of the polynomials $A(x)$ and $B(x)$ are extracted from the signed maximal minors of  $\M_{\alpha,\beta,\overline{\bfn}}(\bfu,\bfv).$  The image of this map is clearly contained in $\pi(W_{\bfn,\alpha,\beta}),$ and the first coordinates of the map define an inverse, so it will be birational with an open subset of $\pi(W_{\bfn,\alpha,\beta})$ provided that we can show that it is regular in a nonempty open subset of $\cU_0.$ But for this to happen, we need that $\prod_{\overline{n_i}<n_i}^nB(u_i)\neq0.$ This holds thanks to Proposition \ref{B(ui)}, hence the map \eqref{mmap} is regular in $\cU_0\cap\{\prod_{\overline{n_i}<n_i}^nB(u_i)\neq0\},$ a  nonzero empty set of $\big(K^l\setminus Z\big)\times K^{\alpha+\beta+1}.$ This concludes with the proof of the Theorem.
\end{proof}
\smallskip
Next result will help us characterize the set of unattainable points for the RHIP.
\begin{corollary}\label{26}

With notation as in the proof of Theorem \ref{osu}, if  $\alpha=0$ or $\beta=0,$  $\pi(V(I_1))=\emptyset,$  otherwise it is a rational variety of dimension $l+\min\{n,\alpha+\beta\}.$
\end{corollary}
\begin{proof}
From \eqref{v} and \eqref{vi}, we deduce that if $\pi(V(I_1))$ is not empty, then it is birational to either $K^2\times\pi\left(  W_{n_2,\ldots, n_l,\alpha-1,\beta-1}\right)$ or $K\times \pi\left( W_{n_1-1,n_2,\ldots, n_l,\alpha-1,\beta-1}\right).$ The claim now follows straightforwardly by applying Theorem \ref{pret} to these varieties.
\end{proof}
\smallskip
For $i=1,\ldots, l,$ let $I_i:= I+\langle B(U_i)\rangle,$ and 
$\M^i_{\alpha,\beta,\bfn}(\bfU,\bfV)$ be the $(n-1)\times (\alpha+\beta+2)$ submatrix of $\M_{\alpha,\beta,\bfn}(\bfU,\bfV)$ made by removing the last row in the block containing the  $U_i$'s. 
\begin{proposition}
For $\alpha,\,\beta\geq1,\,i=1,\ldots, l,$ we have
$$(\bfu,\bfv)\in\pi(V(I_i))\cap (Z\times K^n)\iff\dim_K\big(\mbox{\rm ker}(\M^i_{\alpha-1,\beta-1,\bfn}(\bfu,\bfv))\big)\geq1.$$
\end{proposition}
\begin{proof}
The matrix of the homogeneous linear system given by \eqref{oski} and \eqref{dtt} to compute the coefficients of $(\tilde{A},\tilde{B})$ is actually $\M^i_{\alpha,\beta,\bfn}(\bfU,\bfV).$ The claim now follows straightforwardly.
 \end{proof}

We now come back to the Rational Interpolation Problem.
\begin{theorem}\label{oteo}
For $j=0,\ldots,m+1,\, V(\cF^{j-1}_{k,\bfn})=\pi(W_{\bfn,k-j-1,n-k-j}).$  It is a rational irreducible variety defined over $\K,$ of codimension $2j$ if $j\leq m,$ or the empty set if $j=m+1.$
\end{theorem}
\begin{proof}
From Corollary \ref{fp}, we have that $(\bfu,\bfv)\in V(\cF^{j-1}_{k,\bfn})$ if and only if the matrix  $\M_{k-j-1,n-k-j,\bfn}$ is rank deficient. From \eqref{axx} we deduce then that $$\pi(W_{\bfn,k-j-1,n-k-j})=V(\cF^{j-1}_{k,\bfn}).$$  The rest of the claim now follows from Theorem \ref{pret}.
\end{proof}
\smallskip
\begin{theorem}\label{bp}
Let $1\leq k\leq n.$
The set of unattainable points for the RHIP is a union $\cB_1\sqcup\cB_3\sqcup\ldots\sqcup \cB_{2m-1}$, where for $j=1,\ldots, m,\,\cB_{2j-1}$ is a disjoint union of $l$ rational irreducible varieties defined over $\K$ of codimension $1$ in $V(\cF^{j-2}_{k,\bfn})\setminus V(\cF^{j-1}_{k,\bfn})$ (i.e. of codimension $2j-1$ in the ambient space).
\end{theorem}
\begin{proof}
For $j\in\{1,\ldots, m+1\},$ thanks to Theorem \ref{oteo} we have that
$$V(\cF^{j-2}_{k,\bfn})\setminus V(\cF^{j-1}_{k,\bfn})=\pi(W_{\bfn,k-j,n-k-j+1})\setminus\pi(W_{\bfn,k-j-1,n-k-j}).$$
Let $\cB_{2j-1}$ be the set of unattainable points of the RHIP lying in $V(\cF^{j-2}_{k,\bfn})\setminus V(\cF^{j-1}_{k,\bfn}).$ If  $(\bfu,\bfv)\in\cB_{2j-1},$  there must be $i\in\{1,\ldots, l\}$ such that $(\bfu,\bfv)\in\pi(W^i_{\bfn,k-j,n-k-j+1}),$ where
$$W^i_{\bfn,k-j,n-k-j+1}= W_{\bfn,k-j,n-k-j+1}\cap V(B(U_i)).$$
We deduce that $j$ must be at most $m+1.$ Corollary \ref{26} implies that each of the $\pi(W^i_{\bfn,k-j,n-k-j+1}),\,i=1,\ldots, n,$ is rational (in particular irreducible),  of dimension $2(n-j)+1.$ This concludes with the proof of the Theorem.
\end{proof}
\smallskip
Based on Theorems \ref{oteo} and \ref{bp}, we can design an incremental algorithm to decide the solvability of the RHIP without computing any element of the kernel of a matrix.

\begin{algorithm}\label{algop}$^{}$
\par\noindent\underline{Input:} $\bfn\in\N^l,\,k\in\N,\,(\bfu,\bfv)\in\big(\K^l\setminus Z)\times\K^n$
\par\noindent\underline{Output:}  A message saying that $(\bfu,\bfv)$ is unattainable or not for the RHIP associated to $\bfn.$
\end{algorithm}

\begin{enumerate}
\item $j:=\min\{k-1,|\bfn|-k\}.$
\item Compute the matrix $\M_{ k-1-j, n-k-j,\bfn}(\bfu,\bfv).$
\item If $\mbox{rank}(\M_{ k-1-j, n-k-j,\bfn}(\bfu,\bfv))=n-2j+1,$ then $j-1\mapsto j$, and goto (2).
\item $i:=1.$
\item If $i=l+1$ then print ``not unattainable'' and stop the algorithm.
\item Compute the $(n-1)\times (n-2j-1)$ submatrix $\M^i_{ k-1-j, n-k-j,\bfn}(\bfu,\bfv)$ by removing the  last row in the block indexed by $u_i$', and the columns $k-j$ and $n-2j+1$ of $\M_{k-1-j, n-k-j,\bfn}(\bfu,\bfv).$
\item If $\mbox{rank}(\M^i_{ k-1-j, n-k-j,\bfn}(\bfu,\bfv))=n-2j-1,$ then $i+1\mapsto i$ and goto (5), else print ``unattainable''.
\end{enumerate}

\begin{remark}
One could add an extra step at the end of (3) in Algorithm \ref{algop} to compute a nontrivial vector in the kernel of $\M_{k-1-j, n-k-j,\bfn}(\bfu,\bfv)$ whose coordinates will encode the minimal solution of the WHIP.
\end{remark}

\bigskip
\section{Equations and Proof of the main Theorems}\label{EPF}
In this section we will give explicit equations for the varieties $V(\cF^{j-1}_{k,\bfn})$ and descriptions for the open sets $V(\cF^{j-1}_{k,\bfn})\setminus V(\cF^j_{k,\bfn}).$
They will arise  as maximal minors  of matrices of the form $\M_{\alpha,\beta,\bfn}(\bfU,\bfV),$ with $\alpha+\beta=n-1.$ We start by fixing $(\bfu,\bfv)\in(\K^l\setminus Z)\times\K ^n.$ Recall from the
Introduction that, for $i=1,\ldots, n+1,$ and $1\leq k\leq n,$  we  denote by $\Delta^{\bfn}_{k,i}$ the $i$-th maximal signed minor of
$\M_{k-1,n-k,\bfn}(\bfU,\bfV).$  The following result is an easy verification.

\begin{lemma}\label{deltas}
Up to a sign, $\Delta^{\bfn}_{k-1,n+1}=\Delta^{\bfn}_{k,k}$.
\end{lemma}

\begin{proposition}
\label{dim1}
\[
\dim_\K(\V_{\bfu,\bfv,\bfn,k})=1 \Leftrightarrow \Delta^{\bfn}_{k,k}(\bfu,\bfv)\neq 0 \textrm{ or } \Delta^{\bfn}_{k,n+1}(\bfu,\bfv)\neq 0.
\]
\end{proposition}

\begin{proof}
From Theorem \ref{hipp} we deduce  that $\dim_\K(\V_{\bfu,\bfv,\bfn,k})=1$ if and only if the rank of $\M_{\bfu,\bfv,\bfn,k}$ is maximal, so $\Leftarrow]$ is
clear. Reciprocally, if $\dim_\K (\V_{\bfu,\bfv,\bfn,k})=1,$  then by Hilbert-Burch $(\Delta^{\bfn}_{k,1},\dots , \Delta^{\bfn}_{k,k}; \Delta^{\bfn}_{k,k+1},\dots , \Delta^{\bfn}_{k,n+1})(\bfu,\bfv)$
generates  $\mbox{\rm ker}(\M_{\bfu,\bfv,\bfn,k}),$ and hence must encode a minimal element in $\V_{\bfu,\bfv,\bfn,k}.$ We must
have $\Delta^{\bfn}_{k,k}(\bfu,\bfv)\neq 0$ or   $\Delta^{\bfn}_{k,n+1}(\bfu,\bfv)\neq 0$ thanks to Theorem \ref{hipp} again.
\end{proof}
\smallskip
From Lemma \ref{deltas} we can express Proposition \ref{dim1} in a more symmetric way.

\begin{corollary}\label{43}
\[
\dim_\K (\V_{\bfu,\bfv,\bfn,k})=1 \Leftrightarrow \Delta^{\bfn}_{k,k}(\bfu,\bfv)\neq 0 \textrm{ or } \Delta^{\bfn}_{k+1,k+1}(\bfu,\bfv)\neq 0.
\]
\end{corollary}

\begin{proposition}
\label{dim2} If $2\leq k \leq n-1,$ then
\[
\dim_\K (\V_{\bfu,\bfv,\bfn,k})=2 \Leftrightarrow \left\{ \begin{array}{l} \Delta^{\bfn}_{k,k}(\bfu,\bfv)=\Delta^{\bfn}_{k,n+1}(\bfu,\bfv)=0, \ \mbox{and} \\ \Delta^{\bfn}_{k-1,k-1}(\bfu,\bfv)\neq 0 \textrm{ or } \Delta^{\bfn}_{k+1,n+1}(\bfu,\bfv)\neq 0. \end{array} \right.
\]
\end{proposition}

\begin{proof}
Suppose first that $\dim_\K (\V_{\bfu,\bfv,\bfn,k})=2$. Then, $\Delta^{\bfn}_{k,k}(\bfu,\bfv)=\Delta^{\bfn}_{k,n+1}(\bfu,\bfv)=0$ by
Proposition \ref{dim1}. Let $(a_0,\dots,a_{k-2},0;b_0,\dots, b_{n-k-1},0)$ be the vector of coefficients of a minimal element
in $\V_{\bfu,\bfv,\bfn,k}$, with $a_{k-2}\neq 0$ or $b_{n-k-1}\neq 0$. If $a_{k-2}\neq 0$ then $(a_0,\dots,a_{k-2};b_0,\dots, b_{n-k-1},0,0)$ gives a
minimal element in $\V_{\bfu,\bfv,\bfn,k-1,n}$ and so, by Proposition \ref{dim1} again, and using that $\Delta^{\bfn}_{k-1,n+1}(\bfu,\bfv)=\Delta^{\bfn}_{k,k}(\bfu,\bfv)=0$, we have that $\Delta^{\bfn}_{k-1,k-1}(\bfu,\bfv)\neq 0$.
\par
If $b_{n-k-1}\neq 0$ then $(a_0,\dots,a_{k-2},0,0;b_0,\dots, b_{n-k-1})$ gives a minimal element in $\V_{\bfu,\bfv,\bfn,k+1,n}$ and
now $\Delta^{\bfn}_{k+1,n+1}(\bfu,\bfv)\neq 0$.
\par
Reciprocally, if we assume $\Delta^{\bfn}_{k,k}(\bfu,\bfv)=\Delta^{\bfn}_{k,n+1}(\bfu,\bfv)=0$ and $\Delta^{\bfn}_{k-1,k-1}(\bfu,\bfv)\neq 0,$ then
$\dim_\K (\V_{\bfu,\bfv,\bfn,k})>1$ by Proposition \ref{dim1}. Let $(a_0,\dots ,a_{k-2},0;b_0,\dots , b_{n-k-1},0)$ be the coefficients of a
minimal element in $\V_{\bfu,\bfv,\bfn,k}$. We sort them to produce  $(a_0,\dots ,a_{k-2};b_0,\dots , b_{n-k-1},0,0),$ a minimal
element of $\V_{\bfu, \bfv,\bfn, k-1,n}$, a vector space of dimension $1$ since $\Delta^{\bfn}_{k-1,k-1}(\bfu,\bfv)\neq 0$, and so
$a_{k-2}\neq 0$. By Theorem \ref{hipp} we deduce then that $\dim_\K(\V_{\bfu,\bfv,\bfn,k})=2$.
\par  For the case $\Delta^{\bfn}_{k,k}(\bfu,\bfv)=\Delta^{\bfn}_{k,n+1}(\bfu,\bfv)=0$ and $\Delta^{\bfn}_{k+1-1,n+1}(\bfu,\bfv)\neq 0,$
the coefficients of a minimal element $(a_0,\dots ,a_{k-2},0;b_0,\dots , b_{n-k-1},0)$ in $\V_{\bfu,\bfv,\bfn,k}$ will give a minimal element of the form
$(a_0,\dots ,a_{k-2},0,0;b_0,\dots , b_{n-k-1})$ of the $1$-dimensional vector space  $\V_{\bfu,\bfv,\bfn,k-1,n}$ and hence $b_{n-k-1}\neq 0$.
\end{proof}
\smallskip
From Lemma \ref{deltas} we get the following equivalent result.
\begin{corollary}
\label{dim2bis} If $2\leq m+1$, then
\[
\dim_\K (\V_{\bfu,\bfv,\bfn,k})=2 \Leftrightarrow \left\{ \begin{array}{l} \Delta^{\bfn}_{k,k}(\bfu,\bfv)= \Delta^{\bfn}_{k+1,k+1}(\bfu,\bfv)=0,\  \mbox{and} \\ \Delta^{\bfn}_{k-1,k-1}(\bfu,\bfv)\neq 0 \textrm{ or } \Delta^{\bfn}_{k+2,k+2}(\bfu,\bfv)\neq 0. \end{array} \right.
\]
\end{corollary}

\begin{proposition}\label{imp}
\label{dimj} If $2\leq j\leq m+1$ (that is, $2\leq j\leq k \leq n-j+1$), then
\[
\dim_\K (\V_{\bfu,\bfv,\bfn,k})=j \Leftrightarrow \left\{ \begin{array}{l} \Delta^{\bfn}_{k,k}(\bfu,\bfv)=\dots =\Delta^{\bfn}_{k-j+2,k-j+2}(\bfu,\bfv)=0 \\
\Delta^{\bfn}_{k,n+1}(\bfu,\bfv)= \dots = \Delta^{\bfn}_{k+j-2,n+1} (\bfu,\bfv)=0, \ \mbox{and} \\
\Delta^{\bfn}_{k-j+1,k-j+1}(\bfu,\bfv)\neq 0 \textrm{ or } \Delta^{\bfn}_{k+j-1,n+1}(\bfu,\bfv)\neq 0. \end{array} \right.
\]
\end{proposition}

\begin{proof}

By induction on $j,$ the initial case $j=2$  following from Proposition \ref{dim2}. Let $j>2$.
Assume that $\dim_\K (\V_{\bfu,\bfv,\bfn,k})=j.$ By the induction hypothesis, we must have
 $\Delta^{\bfn}_{k,k}(\bfu,\bfv)=\dots =\Delta^{\bfn}_{k-j+2,k-j+2}(\bfu,\bfv)=0,$ and $\Delta^{\bfn}_{k,n+1}(\bfu,\bfv)= \dots = \Delta^{\bfn}_{k+j-2,n+1}(\bfu,\bfv) =0,$ as otherwise
 it would be equivalent to $\dim_\K(\V_{\bfu,\bfv,\bfn,k})=j_0<j.$

Let $(a_0,\dots , a_{k-j},0, \dots,0; b_0,\dots , b_{n-k-j+1},0\dots,0)$ be the coefficients of a minimal element of this space. If $a_{k-j}\neq 0,$ then
\[(a_0,\dots , a_{k-j}; b_0,\dots, b_{n-k-j+1},0\dots,0,\dots , 0)\]
encodes the coefficients of a minimal element in $V_{\bfu,\bfv,\bfn,k-(j-1),n}$, a vector space of dimension $1$ and, by
Proposition \ref{dim1},  $\Delta^{\bfn}_{k-(j-1),k-(j-1)}(\bfu,\bfv)\neq 0$ (note that
$\Delta^{\bfn}_{k-(j-1),n+1}(\bfu,\bfv)=\Delta^{\bfn}_{k-(j-2),k-(j-2)}(\bfu,\bfv)=0$). If
$b_{n-k-j+1}\neq 0,$ then
\[(a_0,\dots , a_{k-j},0\dots,0, \dots,0; b_0,\dots, b_{n-k-j+1})\]
 are the coefficients of a  minimal element in $\V_{\bfu,\bfv,\bfn, k+j-1,n},$ and hence
 $\Delta^{\bfn}_{k+j-1,n+1}(\bfu,\bfv)\neq 0$ in this case.

Assume now that
$\Delta^{\bfn}_{k,k}(\bfu,\bfv)=\dots =\Delta^{\bfn}_{k-j+2,k-j+2}(\bfu,\bfv)=0$, $\Delta^{\bfn}_{k,n+1}(\bfu,\bfv)= \dots = \Delta^{\bfn}_{k+j-2,n+1}(\bfu,\bfv) =0$.
By the induction hypothesis, we have that $\dim_\K (\V_{\bfu,\bfv,\bfn,k}) >j-1$. Let $$(a_0,\dots , a_{k-j},0, \dots,0; b_0,\dots , b_{n-k-j+1},0\dots,0)$$ be
the coefficients of a minimal element of $\V_{\bfu,\bfv,\bfn,k}$. If $ \Delta^{\bfn}_{k-j+1,k-j+1}(\bfu,\bfv)\neq 0$, then the element
  \[(a_0,\dots , a_{k-j}; b_0,\dots, b_{n-k-j+1},0\dots,0,\dots , 0)\] encodes a non zero element of the one dimensional
  space $\V_{\bfu,\bfv,\bfn,k-j+1,n}$ which implies that $a_{k-j}\neq 0$ and therefore that $\dim_\K (V_{\bfu,\bfv,\bfn,k})=j$. If
  $\Delta^{\bfn}_{k+j-1,n+1}(\bfu,\bfv)\neq 0$, then
\[(a_0,\dots , a_{k-j}, 0\dots,0,\dots , 0; b_0,\dots, b_{n-k-j+1})\] corresponds to a non trivial element in the one-dimensional vector
space $ \V_{\bfu,\bfv,\bfn,k+j-1,n}$ and so $\dim_\K (\V_{\bfu,\bfv,\bfn,k})=j$ also in this case.

\end{proof}
\smallskip
By using again Lemma  \ref{deltas}, we get the following equivalent result.
\begin{corollary}\label{corimp}
\label{dimjbis}  If $2\leq j\leq m+1$, then
\[
\dim (\V_{\bfu,\bfv,\bfn,k})=j \Leftrightarrow \left\{ \begin{array}{l} \Delta^{\bfn}_{k,k}(\bfu,\bfv)=\dots =\Delta^{\bfn}_{k-j+2,k-j+2}(\bfu,\bfv)=0 \\
\Delta^{\bfn}_{k+1,k+1}(\bfu,\bfv)= \dots = \Delta^{\bfn}_{k+j-1,k+j-1}(\bfu,\bfv) =0, \ \mbox{and} \\ \Delta^{\bfn}_{k-j+1,k-j+1}(\bfu,\bfv)\neq 0 \textrm{ or } \Delta^{\bfn}_{k+j,k+j}(\bfu,\bfv)\neq 0. \end{array} \right.
\]
\end{corollary}

We summarize our results with the following
\begin{theorem}\label{mss}
For $j=1,\ldots m+1,\,V(\cF^{j-2}_{k,\bfn})\subset(\K^l\setminus Z)\times \K^n$ is  given by the equations
\begin{equation}\label{dsp}
\Delta^{\bfn}_{k-j+2,k-j+2}=\Delta^{\bfn}_{k-j+3,k-j+3}=\dots  = \Delta^{\bfn}_{k+j-1,k+j-1} =0.
\end{equation}
The open set  $V(\cF^{j-2}_{k,\bfn})\setminus V(\cF^{j-1}_{k,n,\bfn})$ is defined by cutting the above equations with
$$\{\Delta^{\bfn}_{k-j+1,k-j+1}\neq 0\} \cup\{\Delta^{\bfn}_{k+j,k+j}\neq 0\}.$$
\end{theorem}
\begin{proof}
The second part of the claim follows from From Corollary \ref{corimp} and the characterization of
$V(\cF^{j-1}_{k,\bfn})\setminus V(\cF^{j}_{k,\bfn})$  given in Theorem \ref{bg}. The first part follows straightforwardly by noticing that
$$V(\cF^{j-2}_{k,\bfn})=\left(V(\cF^{j-2}_{k,\bfn})\setminus V(\cF^{j-1}_{k,\bfn})\right)\sqcup \left(V(\cF^{j-1}_{k,\bfn})\setminus V(\cF^{j}_{k,n,\bfn})\right)\sqcup\ldots
$$
and applying the first part of the claim to each of these pieces. This concludes with the proof of the Theorem.
\end{proof}

It turns out that this procedure also helps build minimal solutions of the WHIP as follows:

\begin{proposition}\label{ppear}
If $(\bfu,\bfv)\in V(\cF^{j-2}_{k,\bfn})\setminus V(\cF^{j-1}_{k,\bfn})$ then the following expressions are minimal solutions of the WHIP:
\begin{equation}\label{llave}
\left\{
\begin{array}{lcl}
\left(\sum_{\ell=0}^{k-j}\Delta^{\bfn}_{k-j+1,\ell+1}(\bfu,\bfv)x^\ell;\,\sum_{\ell=k-j+1}^{n-2j+2}\Delta^{\bfn}_{k-j+1,\ell+1}(\bfu,\bfv)x^{\ell-k+j-1} \right) \ & \mbox{if} & \Delta^{\bfn}_{k-j+1,k-j+1}(\bfu,\bfv)\neq0, \\
\left(\sum_{\ell=0}^{k-j}\Delta^{\bfn}_{k+j-1,\ell+1}(\bfu,\bfv)x^\ell;\,\sum_{\ell=k+j-1}^{n}\Delta^{\bfn}_{k+j-1,\ell+1}(\bfu,\bfv)x^{\ell-k-j+1} \right) \ & \mbox{if} & \Delta^{\bfn}_{k+j,k+j}(\bfu,\bfv)\neq0.
\end{array}\right.
\end{equation}
If $\Delta^{\bfn}_{k-j+1,k-j+1}(\bfu,\bfv)=0$ (resp. $\Delta^{\bfn}_{k+j,k+j}(\bfu,\bfv)=0$), the first (resp. second) vector in \eqref{llave} vanishes identically.
\end{proposition}

\begin{remark}
From the previous claim we deduce that, as both vectors in \eqref{llave} are minimal solutions of the same WHIP, up to a non-zero constant they must coincide
in $\{\Delta^{\bfn}_{k-j+1,k-j+1}\neq0\}\cap\{\Delta^{\bfn}_{k+j,k+j}\neq0\}.$
\end{remark}
\begin{proof}[Proof of Proposition \ref{ppear}]
If $\Delta^{\bfn}_{k-j+1,k-j+1}(\bfu,\bfv)\neq0$ then, by computing the maximal minors of the matrix $\M_{\bfu,\bfv,\bfn,k-j+1},$ we deduce that
\begin{equation}\label{abobe}
\left(\sum_{\ell=0}^{k-j}\Delta^{\bfn}_{k-j+1,\ell+1}(\bfu,\bfv)x^\ell;\,\sum_{\ell=k-j+1}^{n}\Delta^{\bfn}_{k-j+1,\ell+1}(\bfu,\bfv)x^{\ell-k+j-1} \right)
\end{equation}
is -up to a constant- a minimal solution of the WHIP with parameters $(k-j+1,n)$ as the kernel of this matrix has dimension one, and hence all solutions must
be multiples of \eqref{abobe}.
As $\Delta^{\bfn}_{k-j+1,n+1}(\bfu,\bfv)=\pm\Delta^{\bfn}_{k-j+2,k-j+2}(\bfu,\bfv)$ thanks to Lemma \ref{deltas}, and the last expression equal to zero due to
Theorem \ref{mss}, we deduce that the second coordinate of \eqref{abobe} has degree $d_0<n-k+j-1.$  Note that  \eqref{abobe} is also a minimal solution of the WHIP
with parameters $(n-d_0,n)$ (if it were not minimal, there would be one of smaller degree which would contradict the minimality of \eqref{abobe} as a solution of
the WHIP with parameters $(k-j+1,n).$ Due to Corollary \ref{43}, we deduce then that either
$\Delta^{\bfn}_{n-d_0,n-d_0}(\bfu,\bfv)\neq0$ or $\Delta^{\bfn}_{n-d_0+1,n-d_0+1}(\bfu,\bfv)\neq0.$ From \eqref{dsp}, we must have that
either $n-d_0\leq k-j+1$ or $n-d_0+1\geq k+j.$ The first one cannot happen as we have $d_0<n-k+j-1$ above, so it should be $d_0\leq n-k-j+1,$ which shows
that \eqref{abobe} is a solution of the WHIP with parameters $(k,n)$ (as the denominator has degree smaller than $n-k$), and moreover, we actually
have $$\sum_{\ell=k-j+1}^{n}\Delta^{\bfn}_{k-j+1,\ell+1}(\bfu,\bfv)x^{\ell-k+j-1}=\sum_{\ell=k-j+1}^{n-2j+2}\Delta^{\bfn}_{k-j+1,\ell+1}(\bfu,\bfv)x^{\ell-k+j-1},$$
so the first part of the claim follows for  $\Delta^{\bfn}_{k-j+1,k-j+1}(\bfu,\bfv)\neq0.$ A similar argument show that \eqref{llave} also holds when  $\Delta^{\bfn}_{k+j,k+j}(\bfu,\bfv)\neq0.$
\par For the second part, suppose that $\Delta^{\bfn}_{k-j+1,k-j+1}(\bfu,\bfv)=0.$ As we also have
$\Delta^{\bfn}_{k-j+1,n+1}(\bfu,\bfv)=\pm\Delta^{\bfn}_{k-j+2,k-j+2}(\bfu,\bfv)=0$ (due to Lemma \ref{deltas} and Theorem \ref{mss}), we claim that
all the maximal minors of $\M_{\bfu,\bfv,\bfn,k-j+1}$ vanish identically, and hence \eqref{abobe} is the zero vector. To see this, if
there is a non trivial minor of this matrix, then \eqref{abobe} would compute a non trivial solution of the kernel of $\M_{\bfu,\bfv,\bfn,k-j+1}$ which has
both leading coefficients (numerator and denominator) vanish. So, by multiplying by a polynomial of degree $1$ the two
polynomials in \eqref{abobe},  we would obtain another vector in $\mbox{\rm ker}\big(\M_{\bfu,\bfv,\bfn,k-j+1}\big)$ linearly independent with it. This implies
that the dimension of this kernel is at least two and hence all the maximal minors of the matrix vanish, a contradiction which concludes with the proof of
the Proposition.
\end{proof}
As in Theorem \ref{bp}, we set $\cB_{2j-1}$ to be the set of unattainable points of the RHIP lying in $V(\cF^{j-2}_{k,\bfn})\setminus V(\cF^{j-1}_{k,\bfn}).$  The
following result gives equations for this set.
\begin{theorem}\label{finn}
With notation as in the statement of Theorem \ref{mtt}, for $1\leq j\leq m,\, \cB_{2j-1}$ is a union of $n$  components, each of them being defined, for a
fixed $i\in\{1,\ldots, l\},$  by cutting the $2(j-1)$ equations from \eqref{dsp} with
$$
\left\{\begin{array}{lcl}
\sum_{\ell=k-j+1}^{n-2j+2}\Delta^{\bfn}_{k-j+1,\ell+1}\,U_i^{\ell-k-j+1}=0& \mbox{if} & \Delta^{\bfn}_{k-j+1,k-j+1}\neq 0\\
\sum_{\ell=k+j-1}^{n}\Delta^{\bfn}_{k+j-1,\ell+1}\,U_i^{\ell-k-j+1}=0& \mbox{if} & \Delta^{\bfn}_{k+j,k+j}\neq 0.
\end{array}\right.
$$
Up to a nonzero constant, these two polynomials coincide in the intersection
of  \eqref{dsp} with $\{\Delta^{\bfn}_{k-j+1,k-j+1}\neq 0\}\cap\{\Delta^{\bfn}_{k+j,k+j}\neq 0\}.$
\end{theorem}

\begin{proof}
Recall from the proof of Theorem \ref{bp} that $\cB_{2j-1}$ is the set of unattainable points of the RHIP lying in
$V(\cF^{j-2}_{k,\bfn})\setminus V(\cF^{j-1}_{k,\bfn}).$  Theorem \ref{mss} gives the description of $V(\cF^{j-2}_{k,\bfn})\setminus V(\cF^{j-1}_{k,\bfn}),$
the unattainable points of the RHIP are those where the denominator vanishes after setting $x\mapsto u_i$ for some $i\in\{1,\ldots, l\},$ which means
that -thanks to \eqref{llave}- we need to add the equation
$\sum_{\ell=k-j+1}^{n-2j+2}\Delta^{\bfn}_{k-j+1,\ell+1}\,U_i^{\ell-k-j+1}=0$ for $\Delta^{\bfn}_{k-j+1,k-j+1}\neq 0,$ or
$\sum_{\ell=k+j-1}^{n}\Delta^{\bfn}_{k+j-1,\ell+1}\,U_i^{\ell-k-j+1}=0$ for $\Delta^{\bfn}_{k+j,k+j}\neq 0.$ This concludes with the proof of the claim.
\end{proof}

\bibliographystyle{alpha}
\def\cprime{$'$} \def\cprime{$'$} \def\cprime{$'$}

\end{document}